\pdfoutput=1
\documentclass[11pt,naturalnames]{article}
\usepackage[T1]{fontenc}

\usepackage[pdftitle={Brownian loops on non-smooth surfaces and the Polyakov-Alvarez formula},
  pdfauthor={Minjae Park, Joshua Pfeffer, and Scott Sheffield},
colorlinks=true,linkcolor=NavyBlue,urlcolor=RoyalBlue,citecolor=PineGreen,bookmarks=true,bookmarksopen=true,bookmarksopenlevel=2,unicode=true,linktocpage]{hyperref}

\usepackage{fourier}

\usepackage[margin=1in]{geometry}
\usepackage[usenames,dvipsnames]{xcolor}

\usepackage{amsmath} 
\usepackage{amssymb}
\usepackage{amsthm}
\usepackage{amsfonts}
\usepackage{graphicx}
\usepackage{enumerate}
\usepackage{bm}
\usepackage{upgreek}

\newcommand{\abs}[1]{\left|#1\right|}

\DeclareMathOperator*{\esssup}{ess\,sup}

\def\Var{\mathrm{Var}}
\def\Cov{\mathrm{Cov}}
\def\area{\mathrm{Vol}}
\def\Vol{\mathrm{Vol}}
\def\diam{\mathrm{diam}}
\def\len{\mathrm{len}}
\def\Len{\mathrm{Len}}

\def\clen{\mathrm{clen}}
\def\bclen{\mathrm{bclen}}

\def\Re{\mathrm{Re}}
\def\Im{\mathrm{Im}}

\def\cen{\mathrm{cen}}

\def\ep{\varepsilon}
\def\mcl{\mathcal}
\def\BB{\mathbb}
\def\cdelta{\beta}
\def\smalldiam{d}
\def\smallDiam{d_*}
\def\smallset{D^{\rho,\ep}}
\def\1{\mathbf{1}}
\def\aconst{{\frac{\chi(M)}{6}}}

\def\Vol{\mathrm{Vol}}
\def\tr{\mathrm{tr}}
\def\Lip{\operatorname{Lip}}
\def\ol{\overline}
\def\wt{\widetilde}
\def\CC{\Lambda}

\newtheorem{theorem}{Theorem}[section]
\newtheorem{prop}[theorem]{Proposition}
\newtheorem{lem}[theorem]{Lemma}

\newtheorem{conj}[theorem]{Conjecture}

\newtheorem{ques}[theorem]{Question}

\newtheorem{defn}[theorem]{Definition}

\theoremstyle{remark}
\newtheorem{remark}[theorem]{Remark}

\numberwithin{equation}{section}

\def\@rst #1 #2other{#1}
\newcommand\MR[1]{\relax\ifhmode\unskip\spacefactor3000 \space\fi
  \MRhref{\expandafter\@rst #1 other}{#1}}
\newcommand{\MRhref}[2]{\href{http://www.ams.org/mathscinet-getitem?mr=#1}{MR#2}}

\def\MR#1{\href{http://www.ams.org/mathscinet-getitem?mr=#1}{MR#1}}

\begin{document}

\setcounter{tocdepth}{2}
\setcounter{page}{1}
\setcounter{section}{0}

\author{
\begin{tabular}{c}Minjae Park\\[-5pt]\small University of Chicago\end{tabular}\;
\begin{tabular}{c}Joshua Pfeffer\\[-5pt]\small Columbia University\end{tabular}\;
\begin{tabular}{c}Scott Sheffield\\[-5pt]\small MIT\end{tabular}}

\title{Brownian loops on non-smooth surfaces\\ and the Polyakov-Alvarez formula}

\date{\today}

\maketitle{}

\begin{abstract}
Let $\rho$ be compactly supported on $D \subset \BB R^2$. Endow $\BB R^2$ with the metric $e^{\rho}(dx_1^2 + dx_2^2)$. As $\delta \to 0$ the set of {\em Brownian loops} centered in $D$ with length at least $\delta$ has measure $$\frac{\area(D)}{2\pi \delta} + \frac{1}{48\pi}(\rho,\rho)_{\nabla}+ o(1).$$
When $\rho$ is smooth, this follows from the classical Polyakov-Alvarez formula. We show that the above also holds if $\rho$ is not smooth, e.g.\ if $\rho$ is only Lipschitz. This fact can alternatively be expressed in terms of heat kernel traces, eigenvalue asymptotics, or zeta regularized determinants. Variants of this statement apply to more general non-smooth manifolds on which one considers all loops (not only those centered in a domain $D$).

We also show that the $o(1)$ error is uniform for any {\em family} of $\rho$ satisfying certain conditions.
This implies that if we {\em weight} a measure $\nu$ on this family by the ($\delta$-truncated) Brownian loop soup partition function, and take the vague $\delta \to 0$ limit, we obtain a measure whose Radon-Nikodym derivative with respect to $\nu$ is $\exp\bigl( \frac{1}{48\pi}(\rho,\rho)_{\nabla}\bigr)$. When the measure is a certain regularized Liouville quantum gravity measure, a companion work~\cite{ang2020brownian} shows that this weighting has the effect of changing the so-called central charge of the surface.

\end{abstract}

% \tableofcontents

\bigskip
\noindent\textbf{Acknowledgments.} We thank Morris Ang, Ewain Gwynne, Camillo De Lellis, Sung-jin Oh, and Peter Sarnak for helpful comments. The authors were partially supported by NSF grants DMS 1712862 and DMS 2153742. J.P. was partially supported by a NSF Postdoctoral Research Fellowship under grant 2002159.

\section{Introduction}
Let us first recall a few standard definitions and observations. On a compact surface with boundary, the heat kernel trace can be written $Z = Z(t) = \textrm{sp}\, e^{t\Delta} = \sum e^{t \lambda_n}$ where $\lambda_n$ are the eigenvalues of the Laplace-Beltrami operator $\Delta$. If $I(s)$ is the number of $-\lambda_n$ less than $s$ then $$\int_0^\infty  e^{-ts} I(s)ds= \int_0^\infty e^{-ts} \sum_{n=0}^\infty  1_{s>-\lambda_n} ds =  \sum_{n=0}^\infty \int_{-\lambda_n}^\infty  e^{-ts}ds = \frac1t \sum e^{t \lambda_n} = Z(t)/t.$$ In other words, $Z/t$ is the Laplace transform of $I$. The asymptotics of $Z$ (as $t \to 0$) are therefore closely  related to the asymptotics of $\lambda_n$ (as $n \to \infty$). Weyl addressed the latter for bounded planar domains $D$ in 1911 \cite{weyl1911asymptotische} (see discussion in \cite{mckean-singer}) by showing $-\lambda_n \sim \frac{2\pi n}{\area(D)}$ as $n \to \infty$ which is equivalent to \begin{equation}\label{eqn::weyllaw} Z \sim \frac{\area(D)}{4 \pi t}
\end{equation} as $t \to 0$. In 1966 Kac gave higher order correction terms for $Z$ on domains with piecewise linear boundaries (accounting for boundary length and corners) in his famously titled ``Can you hear the shape of a drum?''\ which asks what features of the geometry of $D$ can be deduced from $I$ (or equivalently from $Z$) \cite{kac1966can}. McKean and Singer extended these asymptotics from planar domains to smooth manifolds with non-zero curvature \cite{mckean-singer} where the constant order correction term is a certain curvature integral. For two dimensional surfaces, with metric given by $e^{\rho}$ times a flat metric, the integral $\int_\delta^\infty \frac{Z(t)}{t} dt$ turns out to be a natural quantity whose small $\delta$ asymptotics involve a constant order term that corresponds to the Dirichlet energy of $\rho$ (the so-called Polyakov-Alvarez formula, also known as the Polyakov-Ray-Singer or Weyl anomaly formula) \cite{ray-singer, polyakov-qg1, alvarez, sarnak1987determinants,  osgood-phillips-sarnak}. This constant order Dirichlet energy term (which can also be formulated in terms of Brownian loop soups, see below) is the main concern of this paper.

Much of the literature assumes that $\rho$ is smooth and makes regular use of objects like curvature that are not well defined if $\rho$ is not $C^2$. But it is known~\cite{hormander1968spectral} that if $\rho$ is only $C^2$ then Weyl's law still holds, i.e.\ $-\lambda_n \sim \frac{2\pi n}{\area_\rho(D)}$~\cite[Example 4.9]{ambrosio2018short} and Weyl's law has been established in certain less smooth settings as well.\footnote{We remark that there is a general theory of metric measure spaces with the so-called ``Riemannian curvature-dimension'' (RCD) condition, not necessarily confined to conformal changes of flat metrics. They include Ricci limit spaces~\cite{sturm2006geometry, lott2009ricci}, weighted Riemannian manifolds~\cite{grigor2006heat}, Alexandrov spaces~\cite{petrunin2010alexandrov}, and many others. A lower bound on Ricci curvature is a key ingredient of many useful estimates in geometric analysis, so Sturm, Lott and Villani~\cite{sturm2006geometry, lott2009ricci} initiated the study of a class of metric measure spaces with a generalized lower-Ricci-bound condition. This has been an active research topic for the last decade; see~\cite{gigli2018lecture} for an overview. The classical Weyl's law and the short time asymptotics for heat kernels on these non-smooth metric measure spaces still hold~\cite{zhang2017weyl, ambrosio2018short}.
Many aspects of the theory are stable under the pointed measured Gromov-Hausdorff topology; for example eigenvalues, heat kernels, and Green's function converge uniformly~\cite{ding2002heat, zhang2017weyl}, Brownian motions converge weakly~\cite{suzuki2019convergence}, etc. Therefore, Weyl's law holds for any RCD space with a measure that can be reasonably approximated. On the other hand, the short time expansion used to define the functional determinant does not exist in this non-smooth setting, so it is not clear if the zeta regularization procedure is also stable.} The problem is somewhat different when the regularity is below $C^2$, since curvature is no longer well-defined everywhere and the relevant estimates no longer hold pointwise and instead hold in an average sense. The primary purpose of this note is to extend some of the basic results in this subject about the conformal anomaly (the Dirichlet energy of $\rho$) to $\rho$ that are less regular---e.g., only Lipschitz---and to show that the rates of convergence can be made to hold uniformly across certain families of $\rho$ values.

This paper is motivated in part by another work by the authors \cite{ang2020brownian} in which similar results are formulated in terms of the so-called {\em Brownian loop measures} which were introduced in \cite{lawler2004brownian} and are related to heat kernel traces on planar domains in e.g.\ \cite{dubedat-coupling, wang2018equivalent} as well as \cite{ang2020brownian}. The results here are useful in the context of \cite{ang2020brownian} because they strengthen the sense in which one can say that ``decorating'' regularized Liouville quantum gravity surfaces by Brownian loop soups has the effect of changing their central charge. We will formulate the results in this paper solely in terms of Brownian loop measures and their generalizations. (The relationship to heat kernels is explained in \cite{ang2020brownian}.)

In addition to the weaker regularity assumptions and the use of generalized loop measures, there are several smaller differences between our presentation and the classical approach in e.g.\ \cite{mckean-singer}: we work in the conformal gauge throughout and do all our calculations in terms of $\rho$, we index loops by their Euclidean center rather than by a typical point on the loop (which would be more similar to the heat kernel approach), and we establish the Polyakov-Alvarez formula in terms of Dirichlet energy directly rather than first establishing an equivalent curvature integral.

Although we encounter some complexity due to the non-smoothness of $\rho$, we also take advantage of the extra simplicity of the two-dimensional setting, where the manifold is completely determined by a conformal factor.

Finally, we note that there is a great deal of additional work in this area, and we cannot begin to survey it all. For example, reference texts such as \cite{berline2003heat, gilkey2018invariance} explore heat kernel traces in greater generality: dimensions other than $2$, operators other than the Laplacian, etc. Other works extend the behavior known for compact smooth manifolds to specific non-smooth manifolds such as those with conical singularities or boundary corners (which both correspond to logarithmic singularities in $\rho$) \cite{mooers1999heat,  kokotov2013polyhedral, sher2013heat, sher2015conic, aldana2018polyakov, greenblatt2021discrete, kalvin2021polyakov} or to non-compact surfaces \cite{albin2013ricci}. There are also many open problems in this subject, which spans probability, geometry, number theory, mathematical physics, and analysis. We present a few of these questions in Section~\ref{sec::open}. We hope that the techniques and perspectives presented here will facilitate progress on these problems and perhaps also find applications in other contexts where Weyl's formula and the Polyakov-Alvarez term appear.

\subsection{Main result} 

Let $\mcl{L}$ denote the set of zero-centered unit-length loops in $\BB R^2$.
We define the Brownian loop measure in the plane by encoding each loop in the plane as an element of $\BB R^2 \times (0,\infty) \times \mcl L$ and formulating the Brownian loop measure as a measure on this product space.

\begin{defn} \label{defn-loop-space}
We express every loop $L$ in $\BB R^2$ by the triple $(x,t, \ell)$, where 
 \begin{itemize}\item $t = \len(L)$ is the \emph{length} of $L$, where we define the length of a path as the length of its parametrizing interval.
\item $x=\cen(L)$ is the \emph{center} of $L$,  Euclidean center of mass of $L$, which is equal to $t^{-1} \int_0^t L(s)ds$.
\item $\ell$ is the zero-centered unit-length loop $s \mapsto t^{-1/2}(L(ts)-x)$ obtained from $L$ by translating the center to zero and rescaling time by $t^{-1}$ and space by $t^{-1/2}$. 
\end{itemize} 
\end{defn}

\begin{defn}
\label{defn-blm}
We define the \emph{Brownian loop measure}  on $\BB R^2$ as the measure on loops $(x,t,\ell)$ in $\BB R^2$ given by
\begin{equation}\nonumber\label{eqn::mucdef}\frac{1}{2\pi t^2}dx \, dt \, d\ell,\end{equation}
where $dx$ denotes Lebesgue measure on $\BB R^2$, $dt$ is Lebesgue measure on $(0,\infty)$ and $d\ell$ is the probabilistic law of the random loop in $\mcl L$ obtained by first sampling a two-dimensional Brownian bridge on $[0,1]$ and then subtracting its mean.\footnote{Equivalently $d\ell$ on the complex plane is the law of the complex-valued GFF indexed by the unit-length circle---with additive constant chosen to make the mean zero. In particular, $d\ell$ is invariant under rotations of that circle.}
\end{defn}

 The mass of the set of Brownian loops centered in $D$ with size greater than $\delta$ is given by \begin{equation}\label{eqn::areadelta} \int_D \int_{\delta}^{\infty} \frac{1}{2\pi t^2}\, dt \, dx = \frac{\area(D)}{2\pi\delta}. \end{equation}  In particular,~\eqref{eqn::areadelta} implies that no matter how small $\delta$ is, most loops with length $t \in (\delta, \infty)$ have length of order $\delta$: half of them have $t < 2\delta$, ninety-five percent have $t<20\delta$, and so forth. Also, the fact that~\eqref{eqn::areadelta} tends to $0$ as $\delta \to \infty$ informally means that there are very few {\em large} loops centered in $D$.
 
 Our main result describes how this mass of Brownian loops changes when we measure the length of loops with respect to a different metric $e^{\rho} |dz|^2$ on the plane, for  $\rho$ a Lipschitz function supported in $D$.  We begin by defining the length of a Brownian loop in the metric $e^{\rho} |dz|^2$, which we call its \emph{$\rho$-length}. If the loop were a smooth curve, we would compute its $\rho$-length by integrating $e^{\rho/2}$ along the curve.  Since Brownian loops have Hausdorff dimension $2$, we instead define its $\rho$-length by integrating $e^\rho$ along the loop, so that it has the same scale factor as area.

 \begin{defn}
 \label{def-rho-length-vol}
Let $(M,g)$ be a smooth two-dimensional Riemannian manifold, and let $\rho$ be a function on $M$.  We define the $\rho$-length $\len_\rho(L)$ of a loop $L$ as $\int_0^{\len(L)} e^{\rho(L(s))} ds$.  We define the $\rho$-volume form $\Vol_\rho$ as the volume form associated to $(M,e^\rho g)$, and we write $\Vol := \Vol_0$.  
\end{defn}

Except in Theorem~\ref{thm-loop-general-surface} and Section~\ref{sec-loop-general-surface}, we always take $(M,g)$ in Definition~\ref{def-rho-length-vol} to be the Euclidean plane. We first describe the space of functions in the scope of this section.

\begin{theorem} \label{thm::loopconvthm}
Let $D$ be a bounded open subset of $\BB R^2$, and $\Lip(D)$ be the space of real-valued Lipschitz functions that vanish outside of $D$. Suppose that $\mathcal B\subset \Lip(D)$ is a collection of functions that (1) has uniformly bounded Lipschitz constants, and (2) is precompact in $W^{1,1}(D)$.

Then as $\delta \to 0$, the $\mu^{\text{loop}}$-mass of loops centered in $D$  with $\rho$-length at least $\delta$, with respect to the Brownian loop measure,  is given by
\begin{equation}\label{eqn::mainthm} \frac{ \area_\rho(D)}{2\pi\delta} + \frac{1}{48\pi} (\rho, \rho)_\nabla + o(1)\end{equation} with the convergence uniform over $\rho \in \mathcal B$.
\end{theorem}

\begin{remark}[Uniform boundedness]
    \label{rmk-boundedness}
The conditions (1) and (2) imply that the functions in $\mathcal B$ are uniformly bounded.
\end{remark}

\begin{remark}[General $p$]
    \label{rmk-general-p}
Since we require for now that the Lipschitz constant is uniformly bounded and the domain is bounded, precompactness in $W^{1,1}$ is equivalent to precompactness in $W^{1,p}$ for any fixed $p \in (1, \infty)$. In particular, Theorem~\ref{thm::loopconvthm} could have been formulated using precompactness in $W^{1,2}$ instead of $W^{1,1}$.  Let us also remark that the space $W^{1,2}(D)$ is equivalent to the space $H^1(D)$ of $\rho$ for which $(\rho, \rho)_\nabla$ is finite.
\end{remark}

\begin{remark}[Precompactness and uniform equicontinuity]
    \label{rmk-uniform-equicontinuity}
    Recall the Fr\'echet-Kolmogorov theorem (e.g., see~\cite{brezis2011functional}):
    let $D \subset \mathbb R^n$ be a bounded domain, and $1\le p<\infty$. A subset $\mathcal A\subset L^p(D)$ is precompact if and only if $\mathcal A$ is bounded in $L^p(D)$ and
    \begin{equation}\label{eqn-fk-thm}
         \sup_{u\in \mathcal A}\,\int_D \abs{u(x+h)-u(x)}^p \,dx\to 0 \quad\text{as } h\to 0,
    \end{equation}
    where $u$ is extended to the function on $\mathbb R^n$ whose value outside $D$ is zero.
    
    We will use this equivalent characterization of precompactness in some of our proofs, usually referred as the \emph{uniform equicontinuity} condition in $L^p$. In particular, we will apply this, in the case $p=1$, to the set $\mathcal A$ of the gradients of the functions in the set $\mathcal B$ from the statement of Theorem~\ref{thm::loopconvthm}.
\end{remark}

\begin{remark}[The uniform equicontinuity condition is necessary]
    As mentioned above, the precompactness hypothesis is equivalent to a type of uniform equicontinuity hypothesis. This hypothesis---or some similar condition on the functions $\rho \in \mcl B$---is necessary for the conclusion of Theorem~\ref{thm::loopconvthm} (or Theorem~\ref{thm::loopconvthm2}) to hold. Simply requiring all surfaces in $\mathcal B$ to be $C^1$ with a universal bound on $|\nabla \rho|$ would not suffice. For example, in the Theorem~\ref{thm::loopconvthm} setting, $\mathcal B$ could contain a sequence $\rho_1, \rho_2, \ldots$ of $C^1$ functions that converge uniformly to zero with $(\rho_j, \rho_j)_{\nabla}=1$ for each $j$. We can construct such a sequence of functions $\rho_j$ by arranging for $\nabla \rho_j$ to oscillate between fixed opposite values, with the oscillation rate becoming faster as $j \to \infty$.  (A simple example of such a family of functions on the torus $[0,2\pi)^2$ is given by a constant multiple of $\rho_j\Bigl((a,b)\Bigr)) = j^{-1} \sin(ja)$; we can define $\rho_j$ similarly on the planar domain $D$ by tapering the sine functions to zero near the boundary of $D$.) We can also perturb the functions to arrange that $\area_{\rho_j}(D) = \area_\rho(D)$ for all $j$.  This set of functions $\mcl B$ does not satisfy~\eqref{eqn::mainthm}: for any fixed $\delta$, one can easily show that
    \begin{equation} \label{eqn::firstlimit} \lim_{j \to \infty} \mu \{ L: \cen(L) \in D, \, \len_{\rho_j}(L) \geq \delta \}  - \frac{\area(D)}{\delta} =0, \end{equation}
    even though for each fixed $j$ we have
    \begin{equation} \label{eqn::secondlimit}  \lim_{\delta \to 0} \bigl(\mu \{ L: \cen(L)\in D, \, \len_{\rho_j}(L) \geq \delta \} - \frac{\area(D)}{\delta}\bigr) = b/2. \end{equation}
    If the limit in~\eqref{eqn::secondlimit} were uniform in $j$, we could choose a $\delta$ with $\mu \{ L: \cen(L) \in  D, \, \len_{\rho_j}(L) \geq \delta \} > b/4$ for all $j$, and
    \eqref{eqn::firstlimit} would not hold for that $\delta$.\footnote{\label{footnote:w12}One might wonder whether it is enough have $\rho$ in the Hilbert space defined by the inner product $(\rho,\rho)_\nabla$, i.e., the Sobolev space $W^{1,2}(D)=H^1(D)$, with (say) zero boundary conditions. Such a $\rho$ can be nowhere differentiable \cite{serrin1961differentiability}, so one would also need to modify the condition on $\mathcal B$. See Question~\ref{ques-general-ftn}.}
    \end{remark}

We extend this result to general surfaces.  The statement of the theorem involves the notion of the zeta-regularized determinant of the Laplacian, as defined, e.g., in~\cite{alvarez, sarnak1987determinants}.  (However, it is not necessary to understand the definition of the zeta-regularized determinant to follow the proof of Theorem~\ref{thm-loop-general-surface}.)

\begin{theorem}
\label{thm-loop-general-surface}
Let $(M,g)$ be a fixed compact smooth two-dimensional Riemannian manifold, and we let $\mu^{\text{loop}}$ denote the Brownian loop measure on $(M,g)$. Let $K$ be the Gaussian curvature on $M$, let $\Delta$ be the Laplacian associated to $(M,g)$, and let $\det_\zeta' \Delta$ denote its zeta-regularized determinant. Let $\mcl B$ be a family of Lipschitz functions that (1) has uniformly bounded Lipshitz constants, and (2) is precompact in $W^{1,1}(M)$.

Then the $\mu^{\text{loop}}$-mass of loops with $\rho$-length between $\delta$ and $C$ is given by
\begin{align}
\nonumber
&\frac{\Vol_{\rho}(M)}{2\pi\delta}  
- \aconst (\log \frac{\delta}{2} + \upgamma) + \log{C} 
+ \upgamma +  \frac{1}{48\pi} \int_M (\| \nabla \rho \|^2 + 2 K \rho)\Vol(dz) \\ 
% \nonumber
& \qquad
 + \log \Vol(M) - \log \Vol_\rho(M) - \log \det\nolimits_{\zeta}' \Delta + o_\delta(1) + o_C(C^{(-1+\ep)/2}), \label{eqn-p-a}
\end{align}
with the convergence as $\delta\to 0$ and $C\to \infty$ uniform over $\rho\in \mathcal B$, where $\upgamma \approx 0.5772$ is the Euler-Mascheroni constant.
\end{theorem}

For simplicity, we have addressed just the compact manifold case, but one could prove a similar result for manifolds with boundary, see Question~\ref{ques-boundary} where we give a heuristic justification in the preceding paragraph; the resulting expression would include a boundary term which is of order $\delta^{-1/2}$.

Observe that, for smooth $\rho$, the expression in the second line~\eqref{eqn-p-a} of the above expression is equal to  $-\log \det_\zeta' \Delta_\rho$, where $\Delta_\rho$ is the Laplacian associated to $(M,e^\rho g)$.  The expression~\eqref{eqn-p-a} for $-\log \det_\zeta' \Delta_\rho$ is known as the Polyakov-Alvarez formula; see, e.g.,~\cite[Proposition~6.9]{ang2020brownian}. Thus, for smooth $\rho$, Theorem~\ref{thm-loop-general-surface} reduces to a relation between the Brownian loop measure and the zeta-regularized Laplacian determinant, which was shown in \cite[Theorem~1.3]{ang2020brownian}.

In fact, we prove a slight generalization of Theorem~\ref{thm::loopconvthm} in which we consider a more general class of loop measures.  

\begin{defn}
The \emph{expected occupation measure} of a random variable $Z:[0,T] \rightarrow \BB{R}^2$ is the function $\theta:\BB{R}^2 \rightarrow (0,\infty)$ such that, for each measurable set $A \subset \BB{R}^2$, the set $\{t \in [0,T]: Z(t) \in T\}$ has expected Lebesgue measure $\int_A \theta(x) dx$.
\end{defn}

\begin{defn}
\label{defn-blm-gen}
We define a \emph{generalized loop measure} $\mu$ as a measure on loops $(x,t,\ell)$ in $\BB R^2$ given by
\begin{equation}\nonumber\label{eqn::genmucdef}\frac{1}{t^2}dx \, dt \, d\ell,\end{equation}
where $dx$ denotes Lebesgue measure on $\BB R^2$, $dt$ is Lebesgue measure on $(0,\infty)$ and $d\ell$ is an {\em arbitrary} rotationally invariant measure on loops in $\mcl L$ whose expected occupation measure is a Schwartz distribution (but not necessarily Gaussian as for the Brownian loop measure). We denote by $b$ the second central moment of the first---or equivalently, second---coordinate of a  random variable whose density is this expected occupation measure.
\end{defn}

Definition~\ref{defn-blm-gen} is the same as Definition~\ref{defn-blm}, except that we no longer insist that $d\ell$ be the Brownian bridge (and we have removed the $2\pi$ factor as it is less natural for general $d\ell$).
The measure $d\ell$ can be supported on the space of circular loops, square-shaped loops, or outer boundaries of Brownian loops, etc. The $\mu$ from Definition~\ref{defn-blm-gen} need not have the same conformal symmetries as the Brownian loop measure. Even if $d\ell$ is supported on smooth loops (rather than Brownian loops) we parameterize the space of loops as in Definition~\ref{defn-loop-space}, so that $(0,t,\ell)$ represents the loop that traces $\sqrt{t} \ell$ over time duration $t$. In the special case of the Brownian loop measure, $b=1/12$.  (See Proposition~\ref{prop-density}.) The Schwartz distribution assumption in Definition~\ref{defn-blm-gen} does not seem necessary for Theorem~\ref{thm::loopconvthm2} to hold, but we have included it to simplify the calculations in the proof of Proposition~\ref{prop-mean-rho} below.

\begin{theorem} \label{thm::loopconvthm2} Let $D$ and $\mcl B$ be as in Theorem~\ref{thm::loopconvthm}, and let $\mu$ be a generalized loop measure in the sense of Definition~\ref{defn-blm-gen}.  Then as $\delta \rightarrow 0$, the $\mu$-mass of loops centered in $D$ with $\rho$-length at least $\delta$ is given by
\begin{equation}\label{eqn::mainthm2} 
\frac{ \area_\rho(D)}{\delta} + \frac{b}{2} (\rho, \rho)_\nabla + o(1),\end{equation}
with the convergence uniform over $\rho \in \mathcal B$.
\end{theorem}

\subsection{Proof outline}

In this section, we let $\mcl B$ be a fixed collection of functions $\rho$ satisfying the hypotheses of Theorem~\ref{thm::loopconvthm2}. To prove Theorem~\ref{thm::loopconvthm2}, we compare $\len_\rho$ to a simpler notion of the length of a loop with respect to $e^{\rho} |dz|^2$, in which we approximate $e^\rho$ along the loop by its value at the center of the loop.  

\begin{defn}
We define the \emph{center $\rho$-length} $\clen_\rho(L)$ of a loop $L$ in $\BB R^2$ centered at a point $x$ as $\int_0^t e^{\rho(x)} ds = e^{\rho(x)} t$.
\end{defn}

We observe that the cutoff $\clen_\rho(L) = \delta$ corresponds to a unique value of $\len(L)$:

\begin{prop}
\label{prop-cdelta}
Let $\delta>0$, $x \in D$ and $\ell \in \mcl L$, and set $\cdelta := e^{-\rho(x)} \delta$.  The loop $L = (x,t,\ell)$ satisfies $\clen_\rho(L) = \delta$ iff $t = \cdelta$, and $\clen_\rho(L) \geq \delta$ iff $t \geq \cdelta$.
\end{prop}

\begin{proof}
The result follows trivially from the definition of center $\rho$-length.
\end{proof}

Proposition~\ref{prop-cdelta} immediately implies the following trivial analogue of Theorem~\ref{thm::loopconvthm2} for center $\rho$-length.

\begin{prop}
\label{prop-clenarea}
The mass of loops $L$ centered in $D$ with $\clen_\rho \geq \delta$ with respect to the Brownian loop measure is given by
\begin{equation}\label{eqn::clenarea} \int_D \int_{\mathcal L} \int_{\cdelta}^\infty \frac{1}{t^2}\, dt \, d\ell \, dx =  \int_D \cdelta^{-1} dx = \int_D \frac{e^{\rho(x)}}{\delta} dx = \frac{\area_\rho(D)}{\delta}. \end{equation}
\end{prop}

\begin{proof}
The result is an immediate consequence of Proposition~\ref{prop-cdelta}.
\end{proof}

We can therefore restate Theorem~\ref{thm::loopconvthm2} as the assertion that if we change our notion of loop length from $\clen_\rho$ to $\len_\rho$, the $\mu$-mass of loops with length $\geq \delta$ increases by  $\frac{b}{2}(\rho,\rho)_\nabla$, up to an error that is  $o(1)$ as $\delta \to 0$ uniformly in $\rho \in \mcl B$.

We divide the proof of Theorem~\ref{thm::loopconvthm2} into two stages. First, in Section~\ref{sec::firstlemma} we show that, up to a uniform $o(1)$ error, replacing $\clen_\rho$ with  $\len_\rho$ has the effect
of subtracting $\delta^{-1}$ times the average discrepancy between the value of $e^\rho$ along a Brownian loop and the value of $e^\rho$ at its center.

\begin{lem} \label{lem::convolveestimate}
Consider a loop sampled from $\mu$ conditioned to have its center in $D$ and length $\beta$.  Let $X$ denote its center, and let $Z$ denote a point on the loop sampled uniformly with respect to length.  Then the $\mu$-mass of loops $L$ with center in $D$ and $\len_\rho(L) \geq \delta$ is equal to the $\mu$-mass of loops $L$ with center in $D$ and $\clen_\rho(L) \geq \delta$, minus 
\begin{equation}    
    \label{eqn-convolve-estimate}
   \frac{1}{\delta}  \mathbb E[e^{\rho(Z)}-e^{\rho(X)}] + o(1),
\end{equation}
    with the $o(1)$ error tending to 0 as $\delta\to 0$ at a rate that is uniform in $\rho\in \mathcal B$. (In \eqref{eqn-convolve-estimate} the expectation is w.r.t.\ the {\em overall} law of $X$ and $Z$ as described above.)
\end{lem}

We then complete the proof of Theorem~\ref{thm::loopconvthm2}  in Section~\ref{sec::secondlemma} by showing that the quantity~\eqref{eqn-convolve-estimate} equals $\frac{b}{2} (\rho,\rho)_{\nabla}$ up to a uniform $o(1)$ error.

\begin{lem} \label{lem::convolveestimateconverges} The quantity~\eqref{eqn-convolve-estimate} equals $\frac{b}{2} (\rho, \rho)_\nabla$ plus an error term that converges to $0$ as $\delta \rightarrow 0$ uniformly in $\rho \in \mcl B$.
\end{lem}

\section{Loop mass difference vs.\ expected length discrepancy} \label{sec::firstlemma}

In this section, we prove Lemma~\ref{lem::convolveestimate} in three steps. 
\medskip

\noindent
\textit{Step 1: Establishing a length threshold $\alpha > 0$ corresponding to $\rho$-length $\delta$.}
We saw in Proposition~\ref{prop-cdelta} that, with $\cdelta = e^{-\rho(x)} \delta$, we have $\clen_\rho(L) = \delta$ if and only if $t = \cdelta$, and $\clen(L) \geq \delta$ if and only if $t \geq \cdelta$. To prove Lemma~\ref{lem::convolveestimate}, we 
establish a similar result for $\rho$-length.  We will show in Proposition~\ref{prop-alpha} below that, for $x \in D$ and $\ell \in \mcl L$ with the diameter of $(x,\delta,\ell)$ sufficiently small, there exists a threshold $\alpha > 0$ such that $\len_\rho((x,t,\ell)) = \delta$ if and only if $t = \alpha$, and $\len((x,t,\ell)) \geq \delta$ if and only if $t \geq \alpha$.  We note that, unlike $\cdelta$, the threshold $\alpha$ may depend on $\ell$ as well as $x$ and $\delta$.
\medskip

\noindent
\textit{Step 2: Relating the difference in the masses of loops to the quantity $\alpha^{-1}$.}
We saw in Proposition~\ref{prop-clenarea} that the $\mu$-mass of loops with center $\rho$-length $\geq \delta$ can be expressed as an integral of $\cdelta^{-1}$.  In Proposition~\ref{prop-integral} below, we similarly express the $\mu$-mass of loops with $\rho$-length $\geq \delta$ as an integral of $\alpha^{-1}$, plus a uniform $o(1)$ error. This reduces the task of proving Lemma~\ref{lem::convolveestimate} to analyzing the difference of integrands $\alpha^{-1} - \cdelta^{-1}$.
\medskip

\noindent
\textit{Step 3: Expressing the difference $\alpha^{-1} - \cdelta^{-1}$ in terms of a difference in lengths.}
We first express the difference $\alpha - \cdelta$ in terms of a difference between the $\rho$-length and center $\rho$-length of a loop (Proposition~\ref{prop-alpha-curve}).  We then apply this result in Proposition~\ref{prop-alpha-inverse} to derive a similar expression for $\alpha^{-1} - \cdelta^{-1}$, which immediately yields Lemma~\ref{lem::convolveestimate}.
\medskip

Having described the main steps of the proof of Lemma~\ref{lem::convolveestimate}, we now proceed with Step 1---proving the existence of the threshold $\alpha$. As we observed in Proposition~\ref{prop-cdelta}, the existence of the analogous threshold $\cdelta$ for center $\rho$-length is trivial, since for fixed $\rho, x,$ and $\ell$, the function $t \mapsto \clen_\rho(L)$ is linear with slope $e^{\rho(x)}$.  This is not the case for $\rho$-length, so we proceed by showing its derivative as a function of $t$ is positive on a sufficiently large interval.  We first observe that, since the functions $\rho \in \mathcal B$ are uniformly bounded above and below, we can crudely bound the function $t \mapsto \len_\rho(L)$ between two linear functions uniformly in $\rho,x,$ and $\ell$.

\begin{prop}
\label{prop-metric-ratio}
There exists a constant $\CC>0$ such that, for each $\rho \in \mathcal{B}$, $x \in D$ and $\ell \in \mathcal{L}$, the loop $L = (x,t,\ell)$ satisfies  \begin{equation} \label{eqn::metricratio}\CC^{-1} \leq \frac{\len_\rho(L)}{\len(L)} \leq \CC.
\end{equation}
In other words, $\len(L)$ and $\len_\rho(L)$ length agree up to a universal constant factor. 
\end{prop}

\begin{proof}
The lemma follows immediately from the fact that $\rho$ is bounded from above and below by a constant uniform in $\rho \in \mathcal {B}$, as noted in Remark~\ref{rmk-boundedness}.
\end{proof}

    In addition, the collection of functions $\{e^\rho\}_{\rho\in\mathcal B}$ satisfies the same conditions of $\mathcal B$, possibly with different bounds.
\begin{prop}
    \label{prop-exp-lip}
    The functions $e^\rho$ for $\rho\in\mathcal B$ also have uniformly bounded Lipschitz constants and are precompact in $W^{1,1}(D)$.
\end{prop}
\begin{proof}
    Since the functions $\rho\in\mathcal B$ are uniformly bounded, their images are contained in some finite closed interval. The exponential function is Lipschitz on any finite closed interval, so the composition $\exp\circ \rho$ is also Lipschitz. Other conditions are straightforward to check.
\end{proof}

We now apply the uniform boundedness of the Lipschitz constants of $e^\rho$ for $\rho \in \mcl B$ to show that,
when $\diam(L) = \diam(\ell)\sqrt{t}$ is not too large,
the derivative of the function $t \mapsto \len_\rho(L)$ is uniformly close to that of the linear function $t \mapsto \clen_\rho(L)$.

\begin{prop} \label{prop-derivative}
For any $\ep>0$, there exists $\smalldiam=\smalldiam(\ep)>0$ and a family of sets $\{\smallset\}_{\rho\in\mathcal B}$ with $\Vol(\smallset)\le \ep$ such that the following is true. For each $\rho \in \mathcal B$, $x\in D\setminus \smallset$, and $\ell\in \mathcal L$, the derivative \[ \frac{\partial}{\partial t} \len_\rho((x,t,\ell)) \] exists and differs from \[ \frac{\partial}{\partial t} \clen_\rho((x,t,\ell)) = e^{\rho(x)} \] by at most $\ep\, \diam(\ell)\sqrt{t}$ for almost every $t>0$ with $\diam(\ell)\sqrt{t}<\smalldiam$. Furthermore, there exists a constant $\wt \Lambda>0$ such that the previous statement it true for arbitrary $\smalldiam$ and $\smallset=\emptyset$ when we choose $\ep=\wt\Lambda$.
\end{prop}

\begin{proof}   
By Rademacher's theorem, any Lipschitz function is differentiable almost everywhere. In particular, by Proposition~\ref{prop-exp-lip}, there exists some constant $\wt{\Lambda}>0$ that does not depend on $\rho$ such that $\abs{\nabla(e^\rho)}<\frac23\wt\Lambda$ for almost every $x\in D$ for all $\rho\in\mathcal B$, where $\nabla(e^\rho)$ is a weak gradient of $e^\rho$. In addition, as $\mathcal B$ is precompact in $W^{1,1}(D)$, it follow from \eqref{eqn-fk-thm} that there exists some $\smalldiam>0$ such that the set
\begin{equation}
    \label{eqn-smallset}
    \smallset:=\{x\in D: \esssup_{|h| \leq \smalldiam} \abs{(\nabla e^{\rho})(x+h)-\nabla e^{\rho}(x)}>\frac23 \ep\}
\end{equation}
satisfies $\Vol(\smallset)\le\ep$ for each $\rho\in \mathcal B$.

For fixed $\rho\in \mathcal B$, $x\in D$, and $\ell\in \mathcal L$, we can write $\len_\rho((x,t,\ell))$ as $A(\sqrt{t})t$, where  $A(r) = A_{\rho,x,\ell}(r) := \int_0^1 e^{\rho}(r \ell(s)+x) ds$. We express a weak $t$-derivative of $\len_\rho((x,t,\ell))$ in terms of $A$ as
 \begin{equation} \label{eqn::lenderiv} \frac{\partial}{\partial t} \len_\rho((x,t,\ell)) =\frac{\partial}{\partial t} ( A(\sqrt{t}) t) = \frac12 t^{-1/2}  A'(\sqrt{t})t +  A(\sqrt{t})= \frac12  A'(\sqrt{t})\sqrt{t} +  A(\sqrt{t}). \end{equation} 
Since
\begin{align}
 A'(r) 
&= \int_0^1 \ell(s) \cdot (\nabla e^{\rho})(r \ell(s)+x) ds \label{eqn-dA-first}\\
&= \int_0^1 \ell(s) \cdot ((\nabla e^{\rho})(r \ell(s)+x) - (\nabla e^{\rho})(x) ) ds, \label{eqn-dA-second}
\end{align}
we can bound $|A'(\sqrt{t})|$ from above by
\begin{equation*}\label{eqn-dA-bound-all}
\diam(\ell) \sup_{h\in\mathbb R^2} \abs{(\nabla e^{\rho})(x+h)}
\leq \frac23\wt{\Lambda}\,\diam(\ell)
\end{equation*}
for almost every $t>0$, using \eqref{eqn-dA-first}. On the other hand, if $x\in D\setminus \smallset$, we use \eqref{eqn-smallset} and \eqref{eqn-dA-second} to bound $|A'(\sqrt{t})|$ from above by
\begin{equation*}\label{eqn-dA-bound-most}
\diam(\ell) \esssup_{|h| \leq \diam(\ell)\sqrt{t}} \abs{(\nabla e^{\rho})(x+h)-\nabla e^{\rho}(x)}
\leq \frac23\ep \,\diam(\ell)
\end{equation*}
for almost every $t>0$ with $\diam(\ell)\sqrt{t}<\smalldiam$.

Note that \eqref{eqn::lenderiv} gives
\begin{equation*} \label{eqn::derivativebound}
\left| \frac{\partial}{\partial t} \len_\rho((x,t,\ell)) %\int_0^t e^{\rho(L(s))} ds
- A(\sqrt{t}) \right| = \frac12\abs{A'(\sqrt{t})}\sqrt{t},
\end{equation*}
and we also have
\begin{equation}
    \label{eqn-A-estimate}
|A(\sqrt{t}) - e^{\rho(x)}|=|A(\sqrt{t}) - A(0)| \leq \int_0^{\sqrt{t}} |A'(s)| ds \le \sup_{0\le s\le \sqrt{t}}\abs{A'(s)}\sqrt{t}.
\end{equation}
Combining these two inequalities with the estimates for $\abs{A'(\sqrt{t})}$ in two different cases gives the desired result.
\end{proof}

We use Proposition~\ref{prop-derivative} to show that, just as for center $\rho$-length, there is a \emph{unique} value $\alpha$ of $\len(L)$ corresponding to $\len_\rho(L) = \delta$.

\begin{prop}
\label{prop-alpha}
We can choose $\smallDiam>0$ such that the following holds. 
For each $\rho \in \mathcal{B}$ and $x\in D$, if $L = (x,\delta,\ell)$ is a loop with
$\diam(\ell)\sqrt{\delta}$ less than $\smallDiam$, then there is a unique positive $\alpha = \alpha(x, \delta, \ell,\rho)$ such that $\len_\rho((x,t,\ell)) = \delta$ iff $t=\alpha$ and $\len_\rho((x,t,\ell)) > \delta$ iff $t >\alpha$.  (If
$\diam(\ell)\sqrt{\delta}$ is $\geq \smallDiam$, then we arbitrarily set $\alpha = \delta$, so that $\alpha$ is defined for every loop $L=(x,\delta,\ell)$.)
\end{prop}

\begin{proof}
Let $\CC$ be the constant in Proposition~\ref{prop-metric-ratio}. If $t \geq \CC \delta$, then $\len_\rho((x,t,\ell))\geq \CC^{-1}t \geq \delta$ by Proposition~\ref{prop-metric-ratio}.  Thus, it suffices to show that, for
$\diam(\ell)\sqrt{\delta}$ sufficiently small, the function $\len_\rho((x,t,\ell))$ is strictly increasing in $t$ when $t \in (0,\CC \delta)$.  We  prove this fact by analyzing $\frac{\partial}{\partial t} \len_\rho((x,t,\ell))$ using Proposition~\ref{prop-derivative}.

Let $\wt\Lambda$ be the constant in Proposition~\ref{prop-derivative} so that for all $\rho \in \mathcal B$, $x \in D, \ell\in \mathcal L$, and almost every $t>0$,
\begin{equation}
    \label{eqn-derivative-bound}
\left| \frac{\partial}{\partial t} \len_\rho((x,t,\ell)) - e^{\rho(x)} \right| \leq \wt\Lambda\, \diam(\ell)\sqrt{t}.
\end{equation}
We choose
$\diam(\ell)\sqrt{\delta}$  sufficiently small less than $\smallDiam>0$ such that, for almost every $t \leq \CC \delta$, the bound
$\wt\Lambda\, \diam(\ell) \sqrt{t}$ in~\eqref{eqn-derivative-bound} is less than $\frac{1}{2} e^{\rho(x)}$.  This means $\frac{\partial}{\partial t} \len_\rho((x,t,\ell))$ is strictly positive for almost every $t \leq \CC \delta$. Since $\len_\rho((x,t,\ell))$ is a continous function in $t$, we conclude that $\len_\rho((x,t,\ell))$ is strictly increasing on $(0,\Lambda \delta)$.
\end{proof}

When $x$ and $\ell$ are fixed, $\alpha$ gives the Euclidean $t$ value of $\len$ that corresponds to a $\delta$ value of $\len_\rho$. We obtain the following analogue of Proposition~\ref{prop-clenarea} with $\len_\rho$ in place of $\clen_\rho$.
 
\begin{prop}
\label{prop-integral}
The $\mu$-mass of the set of loops $L$ with center in $D$ and $\len_\rho(L) \geq \delta$ is equal to 
 \begin{equation} \label{eqn::alphainverse}
 \int_D \int_{\mcl L} \int_{\alpha}^\infty \frac{1}{t^2}\, dt \, d\ell \, dx = \int_D \int_{\mcl L} \alpha^{-1} \, d\ell \, dx \end{equation}
 plus an error term that tends to $0$ as $\delta \rightarrow 0$ at a rate that is uniform in $\rho \in \mathcal B$.
\end{prop}

The reason the mass of loops does not exactly equal \eqref{eqn::alphainverse} is that $\alpha$ is improperly defined when $\diam(\ell)\sqrt{\delta} \ge \smallDiam$, with $\smallDiam$ as in Proposition~\ref{prop-alpha}. Proposition~\ref{prop-integral} asserts that the resulting error is negligible in the $\delta \to 0$ limit. 

To prove Proposition~\ref{prop-integral}, we use the following bound on the $d\ell$-measure of loops with large diameter.

\begin{prop}
\label{prop-decay-2}
The $d\ell$-measure of loops $\ell$ with diameter greater than $K$ tends to $0$ as $K \rightarrow \infty$ faster than any negative power of $K$.
\end{prop}

\begin{proof}
If $\ell$ has diameter $>K$, then Proposition~\ref{prop-metric-ratio} implies that $\int_0^{\len(\ell)} \mathbf{1}_{|\ell(s)| > K/3} dt > cK$ for some constant $c>0$. It follows from Markov's inequality that the probability of this event is bounded from above by $(cK)^{-1} \BB{E}(\int_0^{\len(\ell)} \mathbf{1}_{|\ell(s)| > K/3} dt)$.  By the Schwartz condition in Definition~\ref{defn-blm-gen}, the latter decays as $K \rightarrow \infty$ faster than any negative power of $K$.
\end{proof}

\begin{proof}[{Proof of Proposition~\ref{prop-integral}}]
By definition of $\alpha$, the integral in~\eqref{eqn::alphainverse} gives the $\mu$-mass of loops $L$ with $\cen(L) \in D$ and either
\begin{itemize}
    \item $\diam(\ell)\sqrt{\delta} < \smallDiam$ and $\len_\rho(L) \geq \delta $, or
    \item  $\diam(\ell)\sqrt{\delta} \geq \smallDiam$ and $t\geq \delta$.
\end{itemize}
Thus, the error term---i.e., the difference between~\eqref{eqn::alphainverse} and the mass of loops we consider in the proposition statement---is the $\mu$-mass of loops $L$ with $\diam(\ell)\sqrt{\delta} \geq \smallDiam$ and for which either 
\begin{itemize} 
\item $t \geq \delta$ and $\len_\rho(L) < \delta$, or
\item $t < \delta$ and $\len_\rho(L) \geq \delta$.  
\end{itemize}
To bound this error, we recall from Proposition~\ref{prop-metric-ratio} that $\len_\rho(L)$ is $>\delta$ when $t>\CC\delta$ and $<\delta$ when $t<\CC^{-1} \delta$.  Thus, the error is at most the $\mu$-mass of loops $L$ with center in $D$,  $\diam(\ell)\sqrt{\delta} \geq \smallDiam$ and $\CC^{-1} \delta \leq t \leq \CC \delta$.  This mass is equal to the Lebesgue measure of $D$ times $(\CC-\CC^{-1})/\delta$ times the $d\ell$-measure of the set of loops $\ell$ with $\diam(\ell)\sqrt{\delta} > \smallDiam$.    From Proposition~\ref{prop-decay-2}, we deduce that the error tends to $0$ as $\delta \rightarrow 
0$ at a rate that is uniform in $\rho \in \mathcal B$.
\end{proof}

Proposition~\ref{prop-integral} reduces the problem of proving Lemma~\ref{lem::convolveestimate} to the problem of analyzing the difference of integrands $\alpha^{-1}$ and $\cdelta^{-1}$ in~\eqref{eqn::alphainverse} and~\eqref{eqn::clenarea}.
We first derive an expression for $\alpha - \cdelta$ in terms of the difference in the $\rho$-length and center $\rho$-length of the loop $(x,\cdelta,\ell)$.

\begin{prop}
\label{prop-alpha-curve}
Let $\ep, K>0$ be fixed. Then, for all $\ell\in\mathcal L$ with diameter at most $K$, as $\delta\to 0$,
\begin{equation}
    \label{eqn-alpha-estimate}
\alpha - \cdelta = e^{-\rho(x)} \Bigl( \len_\rho(x, \cdelta, \ell) - \clen_\rho(x, \cdelta, \ell) \Bigr) + o(\delta^2),
\end{equation}
with the error converging uniformly in the choice of $\rho \in \mathcal B$, $\ell$ with diameter at most $K$, and $x \in D\setminus \smallset$ where $\{\smallset\}_{\rho\in D}$ is defined as in Proposition~\ref{prop-derivative}.  If we remove the restriction on $\diam(\ell)$, then the error is $O(\delta^2 \diam(\ell)^2)$ uniformly in $\rho \in \mcl B$, $x \in D\setminus\smallset$, and $\ell \in \mcl L$..
\end{prop}

\begin{figure}[t!]
\begin{center}
\includegraphics[width=.4\textwidth]{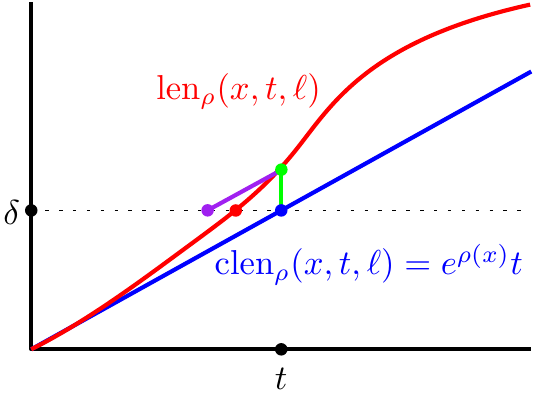}
    \caption{An illustration of the quantities we consider in 
    Proposition~\ref{prop-alpha-curve}.    
    With $\rho\in \mathcal B$, $x\in D\setminus \smallset$, and $\ell\in \mathcal L$ fixed, the blue curve is the graph of $t \mapsto \clen_\rho(L)$, and the red curve is the graph of $t \mapsto \len_\rho(L)$. The blue curve is a line with slope $e^{\rho(x)}$, and the red curve is differentiable with the same derivative $e^{\rho(x)}$ at the origin. Here, we consider $\delta>0$ for which $\diam(\ell) \sqrt{\delta} < \smallDiam$, with $\smallDiam$ as in Proposition~\ref{prop-alpha}. By Proposition~\ref{prop-alpha}, the red curve intersects the horizontal (dotted) line of height $\delta$ at the single red point $(\alpha,\delta)$.  The blue curve intersects the line of height $\delta$ at the blue point $(\beta,\delta)$.  The green point is the point on the red curve directly above the blue point, and the purple line segment is the segment parallel to the blue line from the green point to the dotted line (purple point).  Proposition~\ref{prop-alpha-curve} asserts that the distance between the blue and red points can be approximated by the distance between the blue and purple points.  The latter distance is simply the length of the green segment divided by the slope $e^{\rho(x)}$ of the blue line.} \label{fig::deltaversust}
\end{center}
\end{figure}

\begin{proof}
Let $\smallDiam>0$ be as in Proposition~\ref{prop-alpha}.  Observe that if we are restricting to $\ell$ with diameter at most some constant $K$,  we automatically have $\diam(\ell)\sqrt{\delta} < \smallDiam$ for uniformly small $\delta$.  If we do not impose the restriction  $\diam(\ell) \leq K$, then we could have $\diam(\ell)\sqrt{\delta} > \smallDiam$ for arbitrarily small $\delta$.  However, the proposition statement easily holds for this range of $\delta$ and $\ell$: by Proposition~\ref{prop-metric-ratio}, the error in the proposition statement must be bounded by a uniform constant times $\delta$, which is $O(\delta^2 \diam(\ell)^2)$  when $\diam(\ell)\sqrt{\delta} > \smallDiam$ because $\delta^2 \diam(\ell)^2 > \smallDiam^2 \delta$.  Thus, we may assume for the rest of the proof that $\diam(\ell) \sqrt{\delta} < \smallDiam$.

Throughout the proof, we refer to the graphs of the functions $t \mapsto \clen_\rho(L)$ and $t \mapsto \len_\rho(L)$ for fixed $\rho\in \mathcal B$, $x\in D\setminus \smallset$, and $\ell\in \mathcal L$ in Figure~\ref{fig::deltaversust}. See the caption of the figure for the definitions of the red, blue and green points.  By Proposition~\ref{prop-alpha},  $\alpha - \beta$ is the distance between the red and blue points, and the quantity $e^{-\rho(x)} \left( \len_\rho(x, \cdelta, \ell) -  \clen_\rho(x, \cdelta, \ell) \right)$ on the right-hand side of~\eqref{eqn-alpha-estimate} is the distance between the blue and purple points.  We can express these two distances in terms of the slopes of the two curves:
\begin{itemize}
    \item 
The distance between the blue and purple points is equal to the length of the green segment divided by the slope of the blue line (i.e., $e^{\rho(x)}$). 
\item
The distance between the blue and red points is equal to the length of the green segment divided by the average derivative of the red curve between the red and green points.
\end{itemize}
The error that we need to bound is the difference between these two distances---i.e., the distance between the red and purple points.  By Proposition~\ref{prop-derivative}, the derivative of the red curve between the red and green points differs from the derivative of the blue line by at most $c\, \diam(\ell) \sqrt{\delta}$, where $c$ is a constant bounded uniformly in $\rho \in \mcl B$ and $x\in D\setminus\smallset$ that can be made arbitrarily small for $\diam(\ell) \sqrt{\delta}$ sufficiently small.  By Proposition~\ref{prop-metric-ratio}, this implies that the \emph{inverses} of these two derivatives differ by a uniform constant multiplied by $c\, \diam(\ell) \sqrt{\delta}$.  Moreover,  the length of the green segment is $\len_\rho(x,\cdelta,\ell) - \delta$, and by~\eqref{eqn-A-estimate},
\begin{equation}
    \label{eqn-green-segment-bound}
\len_\rho(x,\cdelta,\ell) - \delta = \cdelta (A(\sqrt{\cdelta}) - e^{\rho(x)}) \leq c\,\diam(\ell)  \cdelta^{3/2},
\end{equation}
with $c$ as above.  Thus, the error term---i.e., the distance between the red and purple points---is bounded by a uniform constant times $c\, \diam(\ell)\sqrt{\delta} \cdot c\, \diam(\ell) \delta^{3/2}= c^2\delta^2\diam(\ell)^2$.  If we do not restrict to $\diam(\ell) \leq K$, the error is  $O(\delta^2 \diam(\ell)^2)$ uniformly in $\rho \in \mcl B$, $x \in D\setminus\smallset$, and $\ell \in \mcl L$.  If we  restrict to $\diam(\ell) \leq K$, then $c \rightarrow 0$ as $\delta \rightarrow 0$ at a uniform rate, so the error is $o(\delta^2)$ uniformly in $\rho \in \mcl B$, $x \in D\setminus\smallset$, and $\ell$ with $\diam(\ell)\le K$.
\end{proof}

Proposition~\ref{prop-alpha-curve} immediately yields the following expression for $\alpha^{-1} - \cdelta^{-1}$.

\begin{prop}
\label{prop-alpha-inverse}
Let $\ep, K>0$ be fixed. Then, for all $\ell$ with diameter at most $K$, as $\delta\to 0$,
\begin{equation}
    \label{eqn-alpha-inverse-estimate}
\alpha^{-1} - \cdelta^{-1} = - \delta^{-1} \cdelta^{-1} \Bigl( \len_\rho(x,\cdelta, \ell) - \clen_\rho(x,\cdelta, \ell) \Bigr) + o(1),
\end{equation}
with the error converging uniformly in the choice of $\rho \in \mathcal B$, $\ell$ with diameter at most $K$, and $x \in D\setminus \smallset$ where $\{\smallset\}_{\rho\in D}$ is defined as in Proposition~\ref{prop-derivative}.  If we remove the restriction on $\diam(\ell)$, then the error is $O(\diam(\ell)^2)$ + $o(\delta^2)$ uniformly in $\rho \in \mcl B$, $x \in D\setminus\smallset$, and $\ell \in \mcl L$.
\end{prop}

\begin{proof}
Throughout the following proof, each $O(\cdot)$ and $o(\cdot)$ error converges uniformly as $\delta \rightarrow 0$ in the choice of $\rho \in \mathcal B$, $x \in D\setminus\smallset$, and $\ell$.  

By the Taylor expansion of $f(r)=1/r$ at $\beta$, we have \[ \alpha^{-1}-\beta^{-1} = -\beta^{-2}(\alpha-\beta)+2\gamma^{-3}(\alpha-\beta)^2\] for some $\gamma$ between $\alpha$ and $\beta$.  
Proposition~\ref{prop-alpha-curve} gives
\[ -\beta^{-2}(\alpha-\beta) = - \delta^{-1} \cdelta^{-1} \Bigl( \len_\rho(x,\cdelta, \ell) - \clen_\rho(x,\cdelta, \ell) \Bigr) + o(1).\]
We now handle the $2\gamma^{-3}(\alpha-\beta)^2$ term. From~\eqref{eqn::metricratio}, we have $\CC^{-1}\le \alpha/\delta \le \CC$; therefore, $\gamma^{-3} = O(\delta^{-3})$. Next, by~\eqref{eqn-green-segment-bound} and Proposition~\ref{prop-alpha-curve},  $(\alpha-\beta)^2 = O(c^2\diam(\ell)^2\delta^3) + o(\delta^4)$, where $c$ is a constant bounded uniformly in $\rho \in \mcl B$ that can be made arbitrarily small for $\diam(\ell) \sqrt{\delta}$ sufficiently small.  
Hence, $\gamma^{-3}(\alpha-\beta)^2 = O(c^2\diam(\ell)^2) + o(\delta^2)$. The latter is $o(1)$ with the restriction on $\diam(\ell)$, and  $O(\diam(\ell)^2) + o(\delta^2)$ otherwise.
\end{proof}

We also give a similar estimate for bookkeeping which might be useful for Question~\ref{ques-trace}.

\begin{prop}
\label{prop-alpha-negative2}
With the same assumption of Proposition~\ref{prop-alpha-inverse} without restriction on $\diam(\ell)$, we have
$\alpha^{-2}-\beta^{-2} = O(c\,\diam(\ell)\delta^{-1/2})$ uniformly in $\rho \in \mcl B$, $x \in D\setminus\smallset$, and $\ell \in \mcl L$, where $c$ is a constant bounded uniformly in $\rho \in \mcl B$ that can be made arbitrarily small for $\diam(\ell) \sqrt{\delta}$ sufficiently small. 
\end{prop}
\begin{proof}
By the Taylor expansion of $f(r)=1/r$ at $\beta$, we have $ \alpha^{-2}-\beta^{-2} = -2\gamma^{-3}(\alpha-\beta)$ for some $\gamma$ between $\alpha$ and $\beta$.  
 From~\eqref{eqn::metricratio}, we have $\CC^{-1}\le \alpha/\delta \le \CC$; therefore, $\gamma^{-3} = O(\delta^{-3})$. Next, by~\eqref{eqn-green-segment-bound} and Proposition~\ref{prop-alpha-curve},  $\alpha-\beta = O(c\,\diam(\ell)\delta^{3/2})$, where $c$ is a constant bounded uniformly in $\rho \in \mcl B$ that can be made arbitrarily small for $\diam(\ell) \sqrt{\delta}$ sufficiently small.  
Hence, $\gamma^{-3}(\alpha-\beta) = O(c\,\diam(\ell)\delta^{-1/2})$. 
\end{proof}

\begin{proof}[{Proof of Lemma~\ref{lem::convolveestimate}}]
Let $\ep>0$ and $\{\smallset\}_{\rho\in D}$ be defined as in Proposition~\ref{prop-derivative} so that $\Vol(\smallset)\le \ep$. By Propositions~\ref{prop-clenarea} and~\ref{prop-integral}, integrating $\alpha^{-1} - \beta^{-1}$ over $x \in D\setminus\smallset$ and $\ell \in \mcl L$ yields the $\mu$-mass of loops $L$ centered in $D$ with $\len_\rho(L) \geq \delta$ minus the $\mu$-mass of loops $L$ centered in $D$ with $\clen_\rho(L) \geq \delta$, up to a uniform $o(1)$ error. 
By Proposition~\ref{prop-alpha-inverse}, for each fixed $x \in D\setminus\smallset$ and $\ell \in \mcl L$, the integrand $\alpha^{-1} - \beta^{-1}$ is equal to 
\begin{equation}
    \label{eqn-alpha-inverse-terms}
- \delta^{-1}  \Bigl( \beta^{-1} \len_\rho(x, e^{-\rho(x)} \delta, \ell) - e^{\rho(x)} \Bigr) 
\end{equation}
plus an error term that has the following limiting behavior as $\delta \rightarrow 0$:
\begin{enumerate}[(a)]
    \item 
    \label{item-error-1}
The error is $O(\diam(\ell)^2) + o(\delta^2)$ uniformly in $\rho \in \mcl B$ and $(x,\ell) \in (D\setminus\smallset) \times \mcl L$.
    \item
    \label{item-error-2}
For each fixed $K>0$, the error is $o(1)$ uniformly in $\rho \in \mcl B$ and $(x,\ell) \in (D\setminus\smallset) \times \mcl L$ with $\diam(\ell) \leq K$.
\end{enumerate}
We now integrate $\alpha^{-1} - \beta^{-1}$ over $x \in D\setminus\smallset$ and $\ell \in \mcl L$ and take the $\delta \rightarrow 0$ limit. The integral of~\eqref{eqn-alpha-inverse-terms} is exactly equal to the term $ \frac{1}{\delta}  \mathbb E[e^{\rho(Z)}-e^{\rho(X)}]$ in~\eqref{eqn-convolve-estimate}, so we just need to show that the integral of the error term in the expression for $\alpha^{-1} - \beta^{-1}$ tends to $0$ as $\delta \rightarrow 0$ uniformly in $\rho \in \mcl B$.  

To analyze the integral of this error term, we partition the domain of integration. 
For any function $K(\delta)$ of $\delta>0$, we can partition $(D\setminus \smallset) \times \mcl L$ into two subdomains: the set of pairs $(x,\ell)$ with $\diam(\ell) \leq K(\delta)$, and the set of $(x,\ell)$ with $\diam(\ell) > K(\delta)$.  The bound~\eqref{item-error-1} implies that the integral of the error over $(x,\ell)$ with $\diam(\ell) > k$ equals a uniform constant times $\int_{\{\diam(\ell) > k\}} (\diam(\ell)^2 + \delta^2) d\ell$, which tends to zero as $k \rightarrow \infty$ at a rate uniform in (small) $\delta>0$ and $\rho \in \mcl B$.  Moreover,~\eqref{item-error-2} implies that if $k>0$ is fixed, the integral of the error over $(x,\ell)$ with $\diam(\ell) \leq k$ tends to $0$ as $\delta \rightarrow 0$ uniformly in $\rho \in \mcl B$.  Hence, if we choose a function $K(\delta)$ that tends to infinity sufficiently slowly as $\delta \rightarrow 0$, the integrals of the error term over both subdomains of $(D\setminus \smallset) \times \mcl L$  tend to $0$ as $\delta \rightarrow 0$ uniformly in $\rho \in \mcl B$.

Finally, the integral of $\alpha^{-1} - \beta^{-1}$ over $x\in\smallset$ and $\ell\in\mathcal L$ is bounded by a uniform constant times $\ep\delta$ because of~\eqref{eqn::metricratio}, which finishes the proof.
\end{proof}

\section{Expected length discrepancy vs.\ Dirichlet energy} \label{sec::secondlemma}

We now complete the proof of Theorem~\ref{thm::loopconvthm2} by proving Lemma~\ref{lem::convolveestimateconverges}.  The main step in proving this lemma is proving the following proposition. 
\begin{prop} 
\label{prop-mean-rho}
With $Z$ and $X$ defined as in Lemma~\ref{lem::convolveestimate}, the quantity $\cdelta^{-1} \mathbb E[\rho(Z) - \rho(X)]$ converges to $0$ as $\delta \rightarrow 0$ uniformly in $\rho \in \mcl B$. The expectation here is w.r.t.\ to the overall law of $X$ and $Z$ as defined in Lemma~\ref{lem::convolveestimate}.
\end{prop}

We prove Proposition~\ref{prop-mean-rho} by a Fourier analysis argument.  Throughout this section, we define the Fourier transform of a function $\phi$ as  $\widehat{\phi}(\zeta) := \int_{\mathbb R^2} e^{i\zeta\cdot z}\phi(z)\,dz$ (using the convention of characteristic functions).  To prove Proposition~\ref{prop-mean-rho}, we express the conditional density of $Z$ given $X=x$ in terms of the expected occupation measure of a loop sampled from $\mu$ with given length and center.  In the following proposition, we introduce some notation for this  expected occupation measure and record some of its elementary properties.

\begin{prop}
\label{prop-defn-theta}
For $z \in \BB R^2$ and $s > 0$, let $\theta(z,s)$ be the expected occupation measure of a loop sampled from $\mu$ and conditioned to have length $s$ and center zero. For fixed $s$, the measure $\theta(z,s)$ is radially symmetric  Schwartz, and a probability measure. Each coordinate has second central moment $bs$. Moreover, $\theta$ satisfies the scaling relation $\theta(z/\sqrt{t},s) = t \theta(z, ts)$  for any $t\ge 0$.  Finally, we have  $\lim_{s\to 0}\widehat{\theta}(z,s)=1$.  
\end{prop}

\begin{proof}
Since $\theta(z,s)$ is the law of 
With $\ell$ sampled from $d\ell$, the loop $t \mapsto \sqrt{s} \ell(t/s)$ has law $\theta(z,s)$.  It follows from
Definition~\ref{defn-blm-gen} that  $\theta(z,s)$ is radially symmetric, Schwartz, and a probability measure with the desired second central moments. The scaling relation also follows immediately.  Finally, the second central moments imply that the measures $\theta(z,s)$ converge weakly as $s \rightarrow 0$ to a point mass at $0$, so their characteristic functions converge to $1$.
\end{proof}

We first reduce the task of proving Proposition~\ref{prop-mean-rho} to analyzing the Dirichlet energy of the inverse Laplacian of an appropriately chosen measure.

\begin{prop}
Suppose that $\phi:\mathbb R^2\to \mathbb R$ is a function satisfying \begin{equation}\label{eqn-phi-conditions}
    \Delta \phi(y) = \int_D \cdelta^{-1}(\theta(y-x,\cdelta)-\bm{\delta}(y-x))\,dx \qquad \text{and} \qquad \lim_{\abs{y} \to \infty} \phi(y) = 0,
\end{equation}
where $\theta$ is defined in Proposition~\ref{prop-defn-theta} and $\bm{\delta}$ is the point mass at $0$. Then
\[
\cdelta^{-1} \mathbb E[\rho(Z) - \rho(X)] \leq (\rho,\rho)_{\nabla}^2  (\phi,\phi)_{\nabla}^2 
\]
\end{prop}

\begin{proof}
The conditional density of $Z$ given $X=x$ is $\theta(\cdot-x, \cdelta)$. Therefore,
    \begin{equation}\nonumber%\label{eq::differenceaspairing}
        \cdelta^{-1} \mathbb E[\rho(Z) - \rho(X)]
        = \int_{D} \Bigl(\rho, \cdelta^{-1}\bigl(\theta(\cdot-x,\cdelta)-\bm{\delta}(\cdot-x)\bigr)\Bigr)\,dx
        = (\rho, \Delta \phi)
    \end{equation}
by the definition of $\phi$.  Integrating by parts (or applying Green's identity) gives $(\rho, \Delta \phi) = -(\rho,\phi)_\nabla$, and 
    \begin{equation*}
        (\rho,\phi)_\nabla^2\le (\rho,\rho)_\nabla (\phi,\phi)_\nabla.
    \end{equation*}
by applying the Cauchy-Schwarz inequality.
\end{proof}

Therefore, it is enough to show $(\phi,\phi)_{\nabla}\to 0$ uniformly for the proof. To express a function $\phi$ satisfying~\eqref{eqn-phi-conditions} using Fourier integral\footnote{Alternatively, such $\phi$ can be written in terms of the convolution with Green's function.}, we first define a family of auxiliary functions that we label $\Theta_t$ indexed by $t>0$.

\begin{prop}\label{prop::defofTheta}
    With the notation $\widehat{\theta_s}(\zeta, s_0):= \bigl.\frac{\partial\widehat{\theta}}{\partial s}(\zeta,s)\bigr|_{s=s_0}$, the function $\Theta_t$ for $t> 0$ defined as the inverse Fourier transform of $\abs{\zeta}^{-2}\widehat{\theta_s}(\zeta,t)$ is a radially symmetric Schwartz function. In addition, $\Theta_t(y)=\Theta_1(y/\sqrt{t})/t$ and $\widehat{\Theta_t}(\zeta)$ is uniformly bounded for all $t>0$ and $\zeta\in\mathbb R^2$. 
\end{prop}
\begin{proof}
    Fix $t>0$. Recall that $\widehat{\theta}(r,t)$ is Schwartz in $r\ge 0$, so we use the notation $\widehat{\theta_r}(r, t):= \bigl.\frac{\partial\widehat{\theta}}{\partial \abs{\zeta}}(\abs{\zeta},t)\Bigr|_{\abs{\zeta}=r}$ in this proof.
    The scaling relation and the radial symmetry of $\theta$ (Proposition~\ref{prop-defn-theta}) implies that $\widehat{\Theta_t}(\zeta)=\abs{\zeta}^{-2}\widehat{\theta_s}(\zeta,t) = \widehat{\theta_s}(1,t \abs{\zeta}^2)$. Hence, it is enough to show that $\widehat{\Theta_t}(r)=\widehat{\theta_s}(1,t r^2)$, as a function of $r\ge 0$, is a Schwartz function. This will also imply that $\widehat{\Theta_t}(\zeta) = \widehat{\Theta_1}(\sqrt{t}\zeta)$ is uniformly bounded by $\sup_{t\ge 0}\abs{\widehat{\Theta_1}(t)}$, and $\Theta_t(y)=\Theta_1(y/\sqrt{t})/t$ from the Fourier inversion.
    
    First we repeatedly differentiate both sides of the relation $\widehat{\theta}(1,tr^2) = \widehat{\theta}(r,t)$ using the chain rule, and observe $\partial_s^k\widehat{\theta}(1,tr^2)\mid_{r=0}$ exists as a constant multiple of $\partial_r^{2k}\widehat{\theta}(r,t)\mid_{r=0}$. Also, the first differentiation implies
    \begin{equation*}
        \widehat{\Theta_t}(r) = \widehat{\theta_s}(1,tr^2) = \frac{\widehat{\theta_r}(r,t)}{2tr}.
    \end{equation*}
    Let $P(r)$ be any polynomial in $r$, and $k$ be any nonnegative integer. On one hand, for a fixed $\varepsilon>0$, we have
    \begin{equation*}
        \sup_{r>\varepsilon} \abs{P(r)\frac{d^k}{dr^k}\frac{\widehat{\theta_r}(r,t)}{2tr}}<\infty.
    \end{equation*}
    because $\widehat{\theta_r}(r,t)$ and its derivatives with respect to $r$ are rapidly decreasing,
    On the other hand, note that $\abs{P(r)\frac{d^k}{dr^k}\widehat{\theta_s}(1,t r^2)}$ is continuous on $(0,\varepsilon]$ and bounded at 0. This proves that $\Theta_t$ is Schwartz.
\end{proof}

\begin{prop}
\label{prop-nabla-phi}
With $\Theta_t$ defined in Proposition~\ref{prop::defofTheta}, for $x\in \mathbb R^2$, let $\phi^x:\mathbb R^2\to \mathbb R$ be defined as
    \begin{equation}\nonumber%\label{eq::defofphi}
        \phi^x(y) := -\frac{1}{\delta}\int_0^\delta\Theta_{te^{-\rho(x)}}(y-x)\,dt.
    \end{equation}
    and $\phi:\mathbb R^2\to \mathbb R$ defined as
    \begin{equation}\label{eq::defofphi}
        \phi(y) := \int_D \phi^x(y)\, dx
    \end{equation}
for each $y\in \mathbb R^2$. Then $\phi$ satisfies~\eqref{eqn-phi-conditions}.
\end{prop}
\begin{proof}
    By the Fourier transform, we have
    \begin{align*}
        \widehat{\phi^x}(\zeta) &= -\frac{1}{\delta}\int_0^\delta e^{i\zeta\cdot x} \widehat{\Theta_{te^{-\rho(x)}}}(\zeta) \,dt \\
        &= -\frac{1}{\delta}\int_0^\delta e^{i\zeta\cdot x}\abs{\zeta}^{-2}\widehat{\theta_s}\bigl(\zeta,te^{-\rho(x)}\bigr)\,dt\\
        &= -e^{i\zeta\cdot x}\abs{\zeta}^{-2}\beta^{-1}\bigl(\widehat{\theta}(\zeta, \delta e^{-\rho(x)})-1\bigr),
    \end{align*}
    as we have defined $\beta = \delta e^{-\rho(x)}$.
    Therefore, the Fourier inversion of $\widehat{\Delta \phi^x} = -\abs{\zeta}^2\widehat{\phi^x}$ gives the desired result.
    Furthermore, the first identity shows that $\widehat{\phi^x}\in L^1$ from Proposition~\ref{prop::defofTheta}. By the Riemann-Lebesgue lemma, we conclude that $\lim_{\abs{y}\to \infty}\phi^x(y) = 0$, and thus it follows from the dominated convergence theorem that $\lim_{\abs{y}\to \infty}\phi(y) = 0$.
\end{proof}

\begin{prop}
With $\phi$ defined as~\eqref{eq::defofphi} and $\Theta_t$ defined in Proposition~\ref{prop::defofTheta},
define
    \begin{equation}\label{eq::defofG}
        G_{\rho,r}^y(z) := e^{3\rho(y-rz)/2} \nabla\Theta_1(z e^{\rho(y-rz)/2})
    \end{equation}
for each $\rho\in\mathcal B$, $y\in \mathbb R^2$, $r\in\mathbb R$, and $z\in\mathbb R^2$.
Then
    \begin{equation}\label{eq::defofgradphi}
        \nabla_y \phi(y) = \frac{1}{\delta} \int_0^\delta \frac{1}{\sqrt{t}}\int_{\frac{y-D}{\sqrt{t}}} G_{\rho,\sqrt{t}}^y(z)\,dz \,dt.
    \end{equation}
\end{prop}
\begin{proof}
    From Proposition~\ref{prop::defofTheta},
    \begin{equation}\nonumber
        \phi(y) = -\frac{1}{\delta}\int_0^\delta \Theta_{te^{-\rho(x)}}(y-x)\,dx\,dt
        = -\frac{1}{\delta}\int_0^\delta\int_D \frac{\Theta_{1}\bigl((y-x)(te^{-\rho(x)})^{-\frac12}\bigr)}{te^{-\rho(x)}}\,dx\,dt,
    \end{equation}
    taking gradient and substituting $y-x=\sqrt{t}z$ gives
    \begin{equation}\nonumber
        \nabla_y\phi(y) = \frac{1}{\delta} \int_0^\delta \frac{1}{\sqrt{t}}\int_{\frac{y-D}{\sqrt{t}}} e^{3\rho(y-\sqrt{t}z)/2} \nabla\Theta_1(z e^{\rho(y-\sqrt{t}z)/2})\,dz \,dt,
    \end{equation}
    so the result follows.
\end{proof}

\begin{prop}
\label{prop-F-convergence}
For each $\rho\in\mathcal B$, $r\in \mathbb R$, and $y\in \mathbb R^2$, let
    \begin{equation}\label{eq::defofF}
        F_{\rho, r}(y):= \int_{\mathbb R^2} G_{\rho,r}^y(z)\, dz
    \end{equation}
where $G_{\rho,r}^y$ is defined as~\eqref{eq::defofG}.
Then $F_{\rho, r}$ and $\partial_r F_{\rho, r}$ both converges to 0 in $L^2(\mathbb R^2)$ as $r\to 0$ uniformly over $\rho\in \mathcal B$.
\end{prop}

\begin{proof}
Since $\nabla \Theta_1$ is rapidly decreasing, given any $\varepsilon>0$, there exists $K=K(\epsilon)$ such that 
    \begin{equation}\label{eq::largeKestimate}
        \int_{\abs{z}\ge K} \abs{G_{\rho,r}^y(z)}^2 \,dz <\varepsilon
    \end{equation}
for all $r\in\mathbb R$, $y\in\mathbb R^2$, and $\rho\in \mathcal B$ as $\rho$ is uniformly bounded. By radial symmetry, note that $F_{\rho,0}(y)=0$ for all $\rho$ and $y$ because $\nabla\Theta_1$ is an odd function.
Now we divide the domain of integral~\eqref{eq::defofF} into two regions depending on the sign of $z_1z_2$ when we write $z=(z_1,z_2)\in \mathbb R^2$, so that the integral on each domain becomes
    \begin{equation}\nonumber
        \int_{\pm z_1z_2>0, \abs{z}<K} G_{\rho,r}^y(z)\, dz
        = 
        \int_{z_1>0, \pm z_2>0, \abs{z}<K} \bigl(G_{\rho,r}^y(z)-G_{\rho,-r}^y(z)\bigr)\, dz.
    \end{equation}
Since $\nabla \Theta_1$ is Lipschitz as being Schwartz, the uniform equicontinuity of $\rho$ in $L^2(D)$ (in the sense of Remark~\ref{rmk-uniform-equicontinuity}) implies that $G_{\rho,r}$, as a function of $rz$, is also uniformly equicontinuous in $L^2(\mathbb R^2)$. Therefore, denoting $A_{\pm}=\{z\in\mathbb R^2: z_1>0, \pm z_2>0, \abs{z}< K\}$, there exists $d=d(\varepsilon)>0$ such that
    \begin{align}\nonumber
        \int_D \abs{\int_{\abs{z}<K} G_{\rho,r}^y(z)}^2 dy &\le \sum_{i\in\{\pm\}}\int_{A_i}\int_D \abs{G_{\rho,r}^y(z)-G_{\rho,-r}^y(z)}^2\,dy \,dz
        \le 2\varepsilon
    \end{align}
for all $\rho\in \mathcal B$ whenever $\abs{r}<d$.
Combined with~\eqref{eq::largeKestimate}, we conclude that $F_{\rho,r}$ converges to 0 in $L^2(D)$ uniformly over $\rho\in\mathcal B$ as $r\to 0$.

A similar argument applies to $\partial_r F_{\rho, r}$. In particular, as a function of $\rho$ and $\nabla \rho$, note that $\nabla G_\rho$ is Lipschitz. Hence the uniform equicontinuity of $\rho$ in $W^{1,2}(D)$ (in the sense of Remark~\ref{rmk-uniform-equicontinuity}) implies the uniform equicontinuity of $\nabla G_{\rho, r}^y(z)$ in $L^2(\mathbb R^2)$ over $\rho\in \mathcal B$ as a function of $rz$.
Since $\Theta_1$ is rapidly decreasing, from integration by parts,
    \begin{equation}\nonumber
        \left.\frac{\partial}{\partial r}F_{\rho,r}(y)\right|_{r=0} = -\int_{\mathbb R^2} \nabla_z G_{\rho,0}^y(z)\cdot z\,dz = \int_{\mathbb R^2} 2G_{\rho,0}^y(z)\,dz = 0
    \end{equation}
again by the oddity of $G_{\rho,0}^y$. Repeating the previous argument, we conclude that $\partial_r F_{\rho, r}$ converges to 0 in $L^2(\mathbb R^2)$ uniformly over $\rho\in \mathcal B$ as $r\to 0$.
\end{proof}

\begin{prop}\label{prop::differencePhi}
    For each $y\in\mathbb R^2$, define
\begin{equation}\label{def::Phi}
    \Phi(y):=\frac{1}{\delta}\int_0^\delta \frac{F_{\rho,\sqrt{t}}(y)}{\sqrt{t}}\, dt.
\end{equation}
% \jjosh{What is $C$?}
Then $\nabla \phi- \Phi$ converges to 0 in $L^2(\mathbb R^2)$ uniformly over $\rho\in \mathcal B$ as $\delta\to 0$.
\end{prop}
\begin{proof}
    Given $\ep>0$, let $K=K(\ep)$ be defined such that \eqref{eq::largeKestimate} holds.
    Choose $d(\varepsilon)$ small enough so that as long as $0<r<d$, we have $\abs{G_{\rho,r}^y(z)-G_{\rho,0}^y(z))}<\varepsilon/\pi K^2$ and the diameter of $(y-D)/r$ is greater than $K$.
    Using the triangle inequality, we obtain the desired result from~\eqref{eq::defofgradphi} and~\eqref{def::Phi}.
\end{proof}

\begin{proof}[Proof of Proposition~\ref{prop-mean-rho}]
By the mean value theorem, with $\Phi$ is defined in Proposition~\ref{prop::differencePhi}, we have
\begin{equation}\nonumber
    \int_D\abs{\Phi(y)}^2\,dy \le \frac{1}{\delta}\int_0^\delta \int_D\abs{\frac{F_{\rho,\sqrt{t}}(y)}{\sqrt{t}}}^2\,dy\,dt
    = \int_D\abs{\partial_r {F_{\rho,r}(y)}}^2\,dy
\end{equation}
for some $r\in (0,\sqrt{\delta})$, that is $\Phi$ converges to 0 in $L^2(\mathbb R^2)$ as $\delta\to 0$ uniformly over $\rho\in \mathcal B$ by Proposition~\ref{prop-F-convergence}. With Proposition~\ref{prop::differencePhi}, we conclude $\nabla \phi$ also converges to 0 in $L^2(\mathbb R^2)$ as $\delta\to 0$ uniformly over $\rho\in \mathcal B$. This completes the proof.
\end{proof}

\begin{proof}[Proof of Lemma~\ref{lem::convolveestimateconverges}]
Let $Y = \rho(Z) - \rho(X)$, so that 
\begin{align} \nonumber \delta^{-1} \mathbb E[e^{\rho(Z)} - e^{\rho(X)}]  &=  \delta^{-1}  \mathbb E\left[ e^{\rho(X)} \left( Y + \frac{Y^2}{2!} + \frac{Y^3}{3!} + \ldots \right)\right]  \\
&= \label{eqn::twopart} \delta^{-1} \mathbb E[e^{\rho(X)} Y] 
+\delta^{-1} \mathbb E\left[ e^{\rho(X)} \left(\frac{Y^2}{2!} + O(Y^3) \right)\right] \Bigr)
\end{align}
By Proposition~\ref{prop-mean-rho}, the term $\delta^{-1} \BB{E}[e^{\rho(X)} Y]$ is $o(1)$ with the convergence uniform in $\rho \in \mcl B$.  We now analyze the second term in~\eqref{eqn::twopart}.  Since the gradients $\nabla \rho$ for $\rho \in \mcl B$ are uniformly equicontinuous in $L^1$ (in the sense of Remark~\ref{rmk-uniform-equicontinuity}), we can express $Y$ as $\nabla \rho(X) \cdot (Z-X) + \varepsilon$ almost surely, where $\varepsilon$ is an error term uniformly $o(|Z-X|)$ in expectation, as $|Z-X| \to 0$. Conditional on $X$, each coordinate of $Z-X$ has mean $X$ and variance $\cdelta b$.  Therefore, with $U \sim \text{Uniform}[0,2\pi]$, 
\[
 \BB E((\nabla \rho(X) \cdot (Z-X))^2|X) 
 = |\nabla \rho(X)|^2 \BB E\left[|Z-X|^2 | X\right] \BB{E}(\cos^2(U))
 =  \cdelta b|\nabla \rho(X)|^2.
\]
Therefore,
\begin{align}
\nonumber
    \mathbb E[Y^2|X] &= \mathbb E\left[ (\nabla \rho(X) \cdot (Z-X) + \varepsilon)^2 |X \right]\\ &= 
    \nonumber
    \cdelta b|\nabla \rho(X)|^2 + \mathbb E[2 \varepsilon \nabla \rho(X) \cdot (Z-X)+\varepsilon^2|X] \\ &=  \cdelta b|\nabla \rho(X)|^2 + o(\delta),\nonumber % \label{eqn::Taylorremainder}
\end{align}
with the $o(\delta)$ error uniform in $\rho$ and $X$. We can similarly show that $\mathbb E[Y^3|X]$ is $o(\delta)$ uniformly in $\rho$ and $X$. Multiplying both sides by $\delta^{-1} e^{\rho(X)}$ and taking the expectation, we deduce that the second term in~\eqref{eqn::twopart} is equal to
$\frac{b}{2}(\rho,\rho)_\nabla$ plus an $o(1)$ error that converges uniformly in $\rho \in \mcl B$. 
\end{proof}

\begin{proof}[Proof of Theorem~\ref{thm::loopconvthm2}]  The result immediately follows from combining Proposition~\ref{prop-clenarea}, Lemma~\ref{lem::convolveestimate} and Lemma~\ref{lem::convolveestimateconverges}.
\end{proof}

Finally, to deduce Theorem~\ref{thm::loopconvthm} from Theorem~\ref{thm::loopconvthm2}, we explicitly characterize the expected occupation measure. 

\begin{prop}
\label{prop-density}
For the Brownian loop measure, the expected occupation measure of a loop sampled from $(\mcl L, d\ell)$  is the density of a complex Gaussian  with variance $1/12$.
\end{prop}
\begin{proof}

The law of a loop sampled from $(\mcl L, d\ell)$ is that of a Brownian bridge indexed by the circle minus its mean.  The value of a Brownian bridge indexed by the circle at any given time minus the mean value is a complex mean-zero Gaussian random variable of variance $1/12$; this calculation appears, for example, in \cite{shef-gff}.\footnote{One way to see it is to consider a Gaussian free field $h$ indexed by the circle and observe that the Dirichlet energy on the circle parameterized by $[-1/2,1/2]$ of the function $f(x) = x^2/2$ is given by $1/12$, so $(f,f)_\nabla = 1/12$.  But for a function $g$ on the circle we have from integration by parts that $(f(x), g(x))_\nabla = g(0) - \int_{-1/2}^{1/2} g(x)dx$. Then using the above and the definition of the GFF we have $\Var\bigl((h,f)_\nabla\bigr) = \Var\bigl(h(1/2) - \int_0^1 h(x)dx \bigr) = (f,f)_\nabla=1/12$. By rotational symmetry, this holds if $0$ is replaced by any other number in $[-1/2,1/2]$. The number $1/12$ is also derived in \cite{shef-gff} by Fourier series, and we remark that comparing these two derivations is one way to prove $\sum_{n=1}^\infty n^{-2} = \pi^2/6$. } 
\end{proof}

\begin{proof}[Proof of Theorem~\ref{thm::loopconvthm2}]
The result follows from Theorem~\ref{thm::loopconvthm2} and Proposition~\ref{prop-density}.
\end{proof}

\section{Brownian loops on surfaces}
\label{sec-loop-general-surface}

In this section, we prove Theorem~\ref{thm-loop-general-surface}.  Throughout the section, we let $(M,g)$ be a fixed compact smooth two-dimensional Riemannian manifold, and we let $\mu^{\text{loop}}$ denote the Brownian loop measure on $(M,g)$. We also fix $\mathcal B$ as a precompact set of Lipschitz functions in $W^{1,1}(M)$ with uniformly bounded Lipschitz constants.

To prove Theorem~\ref{thm-loop-general-surface}, we analyze the mass of ``large'' loops and the mass of ``small'' loops with respect to $\rho$-length separately.  We begin by analyzing the mass of large loops by proving a central limit theorem for $\rho$-length along large loops:

\begin{prop}
\label{prop-loop-clt}
We can choose a constant $c > 0$ such that, for every $\ep,t>0$ and $\rho \in \mcl B$ and every loop $L$ sampled from $\mu(z,z;t)$,
\begin{equation}
\label{eqn-loop-clt}
\left| \int_0^t e^{\rho(L(s))} ds - t \Vol_{\rho}(M)/\Vol(M)\right| \leq ct^{(1+\ep)/2}
\end{equation}
with probability $1 - O(e^{-t^{\ep}})$ as $t \to \infty$, with the rate uniform in $\rho \in \mcl B$.
\end{prop}

To prove Proposition~\ref{prop-loop-clt}, we apply the Markov central limit theorem to Brownian motion on $(M,g)$, and then compare Brownian motion to a loop sampled from $\mu(z,z;t)$ by using the following Radon-Nikodym estimate.  

\begin{prop}
\label{prop-rn-derivative}
Let $z \in M$ be fixed, and let $L$ be a loop sampled from $\mu(z,z;t)$.  The  Radon-Nikodym derivative of the law of $L|_{[0,s]}$ with respect to Brownian motion restricted to $[0,s]$ is given by $1 + O(e^{-\alpha (t-s)})$ for some $\alpha = \alpha(M,g) > 0$.  
\end{prop}

\begin{proof}
By Proposition~\ref{prop-001}, $\|\mu(\cdot,z;t-s)/\mu(z,z;t)\|_{\infty} = 1 + O(e^{-\alpha (t-s)})$ for some $\alpha = \alpha(M,g) > 0$, with the error uniform in $z \in M$.  This implies the derivative bound.
\end{proof}

To apply the Markov central limit theorem to Brownian motion on $(M,g)$, we need the following convergence result for the Brownian transition kernel $\mu(\cdot,\cdot;t)$ as $t \to \infty$.

\begin{prop}
\label{prop-001}
We have \[
\|\mu(\cdot,\cdot;t)-(\Vol(M))^{-1}\|_{L^\infty(M \times M)} = O(e^{-\alpha t}).
\]
for some constant $\alpha>0$.
\end{prop}

\begin{proof} 
Let $0=\lambda_0 \leq \lambda_1 \leq \lambda_2 \leq \cdots$ denote the eigenvalues of $\Delta$, and let $\{\phi_n\}$ be a corresponding Hilbert basis of eigenfunctions.  It follows from the heat equation  that $p$ can be written as
\[
p(x,y;t) = \sum_{n=0}^{\infty} e^{-t \lambda_n} \phi_n(x) \phi_n(y).
\]
It is known that $\|\phi_n\|_\infty \leq c \lambda_n^{1/4}$ for some constant $c$ depending only on $(M,g)$~\cite{donnelly2001bounds}. Hence, $\|\mu(\cdot,\cdot;t)-(\Vol(M))^{-1}\|_{L^\infty(M \times M)}$ is bounded from above by a function of $t$ that decays exponentially as $t \to \infty$.
\end{proof}

\begin{proof}[{Proof of Proposition~\ref{prop-loop-clt}}]
(In the proof that follows, the rate of convergence of the $O(\cdot)$ errors are uniform in the choice of $\rho \in \mcl B$.)
Let $B_t$ be a Brownian motion on $(M,g)$ started at a point sampled from the volume measure associated to $(M,g)$.
Let $Y_n = B|_{[n,n+1]}$, and let $f(Y_n) = \int_n^{n+1} e^{ \rho(B_t)} dt$.  For $y \in M$, let $f_*(y)$ be the expected value of $\int_n^{n+1} e^{ \rho(B^y_t)} dt$, where $B_t^y$ is a Brownian motion on $(M,g)$ started at $y$.  
By Proposition~\ref{prop-001}, the conditional expectation of $f(Y_n)$ given $Y_0$ is $\int_M p(B_1,y;n-1) f_*(y) \Vol(dy)=  \mathbb{E}(f(Y_0)) + O(e^{-\alpha n})$ for some $\alpha = \alpha(M,g)$.
Hence, $\Cov(f(Y_0),f(Y_n)) =  O(e^{-\alpha n})$ for some $\alpha = \alpha(M,g)$. Thus, we may apply the Markov chain central limit theorem to deduce that 
\[
\sqrt{n} \left( \frac{1}{n} \sum_{j=1}^n f(Y_j) - \mathbb{E}(f(Y_0)) \right)
\]
converges in the $n \to \infty$ limit to a centered Gaussian distribution whose variance is bounded uniformly in the choice of $\rho \in \mcl B$.  Note that we have $\mathbb{E}(f(Y_0)) = \frac{1}{\Vol(M)}\int_M e^{ \rho(x)}) \Vol(dx) = \Vol_{\rho}(M)/\Vol(M)$.
Therefore, for each fixed $\ep,c>0$,
\[
\abs{\int_0^n e^{\rho(B_t)} dt - n \Vol_{\rho}(M)/\Vol(M)}\leq cn^{(1+\ep)/2} 
\]
with probability $1 - O(e^{-n^{\ep}})$.

Combining this with Proposition~\ref{prop-rn-derivative}, we deduce that, if $n$ is an integer with $t - 2\sqrt{t} \leq n \leq t - \sqrt{t}$, then 
\[
\abs{\int_0^n e^{\rho(L(s))} ds - n \Vol_{\rho}(M)/\Vol(M)} \leq c\sqrt{t}^{(1+\ep)/2}
\]
with probability $1 - O(e^{-t^{\ep}})$ as $t \to \infty$. Since $\int_n^t e^{\rho(L(s))} ds$ is bounded from above by $\sqrt{t}$ times a constant uniform in $\rho \in \mcl B$, this implies the proposition.
\end{proof}

We now apply Proposition~\ref{prop-loop-clt} to analyze the mass of large loops.

\begin{prop}
\label{prop-large-loops}
For each $\ep>0$, the symmetric difference between 
\begin{itemize} \item the set of loops with length $\geq \delta$ and length $\leq C$ , and \item the set of loops with length $\geq \delta$ and $\rho$-length $\leq C \Vol_\rho(M)/\Vol(M)$, \end{itemize} has $\mu^{\text{loop}}$-mass at most $O(C^{(-1+\ep)/2})$, with the rate of convergence uniform in $\rho \in \mcl B$.
\end{prop}
\begin{proof}
(In the proof that follows, the rate of convergence of the $O(\cdot)$ errors are uniform in the choice of $\rho \in \mcl B$.)
Let $\mcl S$ denote the symmetric difference between the two sets.
By Proposition~\ref{prop-metric-ratio}, each loop in $\mcl S$ has length between $\CC^{-1} C$ and $\CC C$, with $\CC$ independent of the choice of $\rho \in \mcl B$.  By Proposition~\ref{prop-loop-clt}, the subset $\mcl S^* \subset \mcl S$ of loops not satisfying~\eqref{eqn-loop-clt} has $\mu^{\text{loop}}$-mass at most $\int_{\CC^{-1} C}^{\CC C} O(e^{-t^{\ep}}/t) dt = O(\exp(-\CC^{-\ep} C^{\ep})/C)$.  Moreover, with $a_{\pm} = C \Vol_{\rho}(M)/\Vol(M) \pm c(C^{(1+\ep)/2}$, we can bound the $\mu^{\text{loop}}$-mass  of the set of loops in  $\mcl S \backslash \mcl S^*$ by 
\[
\int_{a_-}^{a_+} \int_M t^{-1} \mu_\rho(z,z;t) \Vol_{\rho}(dz) dt \leq  \Vol_\rho(M) \int_{a_-}^{a_+} t^{-1}  dt = O(C^{(-1+\ep)/2}).\qedhere
\]
\end{proof}

Next, we analyze the mass of small loops.

\begin{prop}
\label{prop-small-loops}
For each $\rho_1,\rho_2 \in \mcl B$, the difference in the masses of the sets 
\begin{equation}\nonumber
% \label{eqn-loop-difference-1}
\{L : \len_{\rho_j} L \geq \delta, \len L \leq C\}  \qquad j=1,2
\end{equation}
under the Brownian loop measure in $(M,g)$ is equal to the difference in the expressions 
\begin{equation}
\label{eqn-loop-difference-9}
\frac{\Vol_{\rho_j}(M)}{2\pi\delta} +  \frac{1}{48\pi} \int_M (\| \nabla  \rho_j \|^2 + 2 K \rho_j)\Vol(dz),  \quad j=1,2.
\end{equation}
plus a term that tends to zero as $\delta \to 0$ at a rate that is uniform in the choice of $\rho_1,\rho_2 \in \mcl B$.
\end{prop}

To prove Proposition~\ref{prop-small-loops}, we will apply the following pair of propositions.

\begin{prop}
\label{prop-uv}
Let $(M,g)$ be a smooth two-dimensional Riemannian manifold.  Let $U \subset M$, and let $\mcl S$ be a collection of loops in $(M,g)$ that intersects a closed set disjoint from $\ol U$.  Then the following holds for all $\delta>0$ sufficiently small. Let $\rho_1,\rho_2 \in \mcl B$ be uniformly bounded functions that agree outside $U$.  Then, for each $L \in \mcl S$, we have $\len_{\rho_1} L \geq \delta$ iff $\len_{\rho_2} L \geq \delta$.
\end{prop}
\begin{proof}
    It is straightforward as $\mcl B$ is uniformly bounded and the length of $L$ outside $U$ is bounded below.
\end{proof}

\begin{prop}
\label{prop-dirichlet}
Let $f$ be a compactly supported function on a region $D \subset \BB{R}^2$ with finite Dirichlet energy, and let $h$ be a smooth compactly supported function on $D$, and let $\nabla_h,\nu_h,K_h$ denote the gradient, volume form and Gaussian curvature associated to $(D,e^h |dz|^2)$.  Then 
\[
\int_{D} \|\nabla(f+h)\|^2 dz - \int_D \|\nabla h\|^2 dz = \int_{D} (\|\nabla_h f \|^2 + 2 K_h f) \nu_h(dz).
\]
\end{prop}

\begin{proof}
It follows from
\begin{align*}
&\int_{D} (\|\nabla(f+h)\|^2 - \|\nabla h\|^2 ) dz 
=
-\int_{D} (f \Delta f + 2 f \Delta h) dz
\\=
&-\int_{D} (f e^{-h} \Delta f + 2 f e^{-h} \Delta h) \nu_h(dz)
=
-\int_{D} (f \Delta_h f - 2 K_h f) \nu_h(dz)
\\=
&-\int_{D} (f e^{-h} \Delta f - 2 K_h f) \nu_h(dz)
=
\int_{D} (\|\nabla_h f \|^2 + 2 K_h f) \nu_h(dz).\qedhere
\end{align*}
\end{proof}

\begin{proof}[{Proof of Proposition~\ref{prop-small-loops}}]
In the proof that follows, we consider $\rho$-lengths and $\rho$-volume forms (as defined in Definition~\ref{def-rho-length-vol}) for functions $\rho$ on both $(M,g)$ and a region of the Euclidean plane.  To avoid confusion between the two settings, we use the notation $\len_\rho$ and $\Vol_\rho$ in the $(M,g)$ setting, and $\wt \len_\rho$ and $\wt \Vol_\rho$ in the Euclidean setting.

Let $U,V,W$ be open sets in $M$ with $\ol U \subset V$ and $\ol V \subset W$, such that we can find a homeomorphism $\varphi: W \rightarrow \wt W \subset \BB{R}^2$. 
By a partition of unity argument, it suffices to prove the proposition under the assumption that $\rho_1,\rho_2$ agree on $U$. So, we assume that this is the case.
By Proposition~\ref{prop-uv},  the difference in the masses of the sets 
\begin{equation}
\label{eqn-loop-difference-1}
\{L : \len_{\rho_j} L \geq \delta, \len L \leq C\}  \qquad j=1,2
\end{equation}
under the Brownian loop measure in $(M,g)$ is equal to the difference in masses with~\eqref{eqn-loop-difference-1} replaced by
\begin{equation}
\label{eqn-loop-difference-2}
 \{L : \len_{\rho_j} L \geq \delta, \len L \leq C, L \subset V\} \qquad j=1,2.
\end{equation}
Since, by Proposition~\ref{prop-metric-ratio}, $\len L > C$ automatically implies $\len_{\rho_j} L \geq \delta$ for sufficiently small $\delta$ (in a manner that does not depend on the choice of $\rho_1,\rho_2 \in \mcl B$), we can replace~\eqref{eqn-loop-difference-2} by the sets
\begin{equation}
\label{eqn-loop-difference-3}
 \{L : \len_{\rho_j} L \geq \delta, L \subset V\},  \qquad j=1,2,
\end{equation}
with the condition $\len L \leq C$ in~\eqref{eqn-loop-difference-2} removed. 
Let $\wt \rho_j$ be the pushforward of $\rho_j$ via $\varphi$, and let $e^\sigma |dz|^2$ be the pushforward of the metric $g$ via $\varphi$.  
Also, set $\wt V = \varphi(V)$. By conformal invariance of the Brownian loop measure~\cite[Lemma~3.3]{ang2020brownian}, the masses of the sets~\eqref{eqn-loop-difference-3} under the Brownian loop measure in $(M,g)$ equals the masses of the sets
\begin{equation}
\label{eqn-loop-difference-4}
\{L : \wt \len_{\wt \rho_j+\sigma} L \geq \delta, L \subset \wt V\}  \qquad j=1,2
\end{equation}
under the Brownian loop measure in the Euclidean plane.
Now, let $\theta$ be a smooth function on $\BB{R}^2$ that equals $1$ on $\ol{\wt V}$ and zero outside a compact subset of $\wt W$. Since $\theta \equiv 1$ on $\wt V$,~\eqref{eqn-loop-difference-4} is equal to the set 
\begin{equation}
\label{eqn-loop-difference-5}
\{L : \wt \len_{\theta(\wt \rho_j+\sigma)} L \geq \delta, L \subset \wt V\}  \qquad j=1,2
\end{equation}
 By Proposition~\ref{prop-uv}, the difference in these loop masses is unchanged if we replace the sets~\eqref{eqn-loop-difference-5} by the sets
\[
\{L : \wt \len_{\theta(\wt\rho_j+\sigma)} L \geq \delta, \cen(L) \in \wt W\}, \qquad j=1,2.
\]
By Theorem~\ref{thm::loopconvthm}, this difference in loop masses is given by the difference in the quantities
\begin{equation}
\label{eqn-loop-difference-6}
\frac{\wt \Vol_{\theta(\wt\rho_j+\sigma)}(\wt W)}{2\pi\delta} +  \frac{1}{48\pi}  \frac{1}{48\pi} \int_{\wt W} \| \nabla (\theta(\rho_j+\sigma)) \|^2 dz, \qquad j=1,2.
\end{equation}
By Proposition~\ref{prop-dirichlet} with $f=\theta \wt \rho_j$ and $h=\sigma$, together with the fact that $\wt \rho_1 \equiv \wt \rho_2$ outside $U$ and $\theta \equiv 1$ in $V$, we can write the difference in the quantities~\eqref{eqn-loop-difference-6} as the difference in the quantities
\begin{equation}
\label{eqn-loop-difference-7}
\frac{\wt \Vol_{\wt\rho_j+\sigma}(\wt W)}{2\pi\delta} +  \frac{1}{48\pi} \int_{\wt W} (\| \nabla_\sigma \wt \rho_j \|^2 + 2 K_\sigma \wt \rho_j)\wt \Vol_{\sigma}(dz),  \quad j=1,2,
\end{equation}
where $\nabla_\sigma, K_\sigma$ are the gradient and Gauss curvature associated to $(\wt W,e^\sigma |dz|^2)$.
Pulling back via $\varphi$, we can rewrite the expressions~\eqref{eqn-loop-difference-7} as
\begin{equation}
\nonumber
\label{eqn-loop-difference-8}
\frac{\Vol_{\rho_j}(W)}{2\pi\delta} +  \frac{1}{48\pi} \int_W (\| \nabla \rho_j \|^2 + 2 K \rho_j)\Vol(dz),  \quad j=1,2.
\end{equation}
We complete the proof by noting that, since $\rho_1 \equiv \rho_2$ outside $W$, we can replace $W$ by $M$ in~\eqref{eqn-loop-difference-9}.
\end{proof}

\begin{proof}[{Proof of Theorem~\ref{thm-loop-general-surface}}]
By conformal invariance of the Brownian loop measure~\cite[Lemma~3.3]{ang2020brownian}, the statement of the theorem is equivalent to the assertion that~\eqref{eqn-p-a} holds up to scaling.
The result holds for $\rho \equiv 0$ by~\cite[Theorem 1.3]{ang2020brownian}.\footnote{We remark that the $\delta$ in the statement of~\cite[Theorem 1.3]{ang2020brownian} represents the quadratic variation of Brownian loops, which is two times the time interval length we use in this paper. Thus, for our application, we need to substitute $\delta$ there into $\delta/2$.} We deduce the result for general $\rho \in \mcl B$ by applying Propositions~\ref{prop-large-loops} and~\ref{prop-small-loops}.
\end{proof}

\section{Constructing square subdivision regularizations}
Now what happens if $\rho$ is defined from a finite square subdivision as in~\cite[Section 6]{ang2020brownian}, so that it has constant Laplacians on each square in a grid?  These functions can be shown to be $C^1$, but along the edges of the squares they fail to be $C^2$. In this section, we explain enough about these function to make it clear that they fit into the framework of this paper, at least if one restricts attentions to those for which the square averages are restricted to a compact set.

One can construct functions with piecewise-constant Laplacian on squares somewhat explicitly. Consider the function on $f(z) = \Re[z^2 \log(z^2)]/\pi + |z|^2/2$ restricted to the quadrant  $Q = \{z: \arg(z) \in (-\pi/4,\pi/4) \}$. Note that on the boundary of the quadrant, $z^2$ is purely imaginary, equal to $|z^2|i$ on the upper boundary ray and $-|z^2|i$ on the lower, while $\Im \log(z^2) = \arg(z^2)$ is $\pi /2$ on the upper boundary and $- \pi /2$ on the lower. so that $\Re[z^2\log(z^2)]/\pi = -|z^2|/2$ on both rays. Thus $f$ has constant Laplacian and equals zero on the boundary $Q$. We can extend the definition of $f$ to the other three quadrants ($iQ$, $-iQ$, and $-Q$) by imposing the relation $f(iz) = -f(z)$. In a small neighborhood of the origin, the $f$ defined this way is negative on $\pm Q$ and positive on $\pm iQ$. The complex derivatives of $g(z) = z^2 \log(z^2) = 2 z^2 \log(z)$ are first $g'(z) = 4z \log(z) + 2z$, second $g''(z) = 4 \log(z) + 6$, third $g'''(z) = 4/z$, etc.  One can deduce from this that $f$ is differentiable as a real-valued function (with derivative $0$ at $0$) but that its second derivatives blow up slowly (logarithmically) near zero and also have discontinuities along the quadrant boundaries.

The function $f(\omega z)$ (where $\omega$ is a fixed eighth root of unity) is thus a $C^1$ function that has piecewise constant Laplacian on the four standard quadrants, while being equal to zero on the boundaries of these quadrants.

By taking linear combinations of $f(\omega z)$ and the four functions $\bigl[\max\bigl( \Re(a z), 0\bigr)\bigr]^2$ (with $a \in \{\pm 1, \pm i\}$) we can get a differentiable function whose Laplacian matches any function that is constant on each of the four standard quadrants --- in particular a function whose Laplacian is $1$ on one quadrant and $0$ on the other three.  By taking differences of translates of this function, we can obtain a function whose Laplacian is $1$ on a semi-strip or a rectangle (and $0$ elsewhere).  Linear combinations of these allow us to describe a function $\phi$ whose Laplacian is any given function that is constant on the squares of the grid.  Any other function with the same Laplacian has to then differ from $\phi$ by a harmonic function. For example, by subtracting the harmonic extension of the values of $\phi$ on $\partial D$ one obtains a function with the desired Laplacian whose boundary values on $\partial D$ are zero.

\section{Open questions} \label{sec::open}

In this section, we discuss some open questions. First, as already mentioned in Footnote~\ref{footnote:w12}, we expect that some of our main results can be extended to $\rho\in W^{1,2}(D)$. In addition to the fact that the Dirichlet energy is well-defined for such functions, here is another reason that the arguments from the Lipschitz case might carry over to this setting.
Suppose that $-M\le \rho\le M$. Then the function
    \begin{equation*}
        f(x):= \begin{cases}
        e^x-1 & \text{ for } x \in [-M,M]\\
        0 & \text{ otherwise.}
        \end{cases}
    \end{equation*}
    is a Lipschitz function on $[-M,M]$ with $f(0)=0$. Therefore, the celebrated Stampacchia's theorem~\cite[\S 4.2.2]{evans2018measure} asserts that $\rho\in W^{1,2}(D)$ implies $e^{\rho}-1$ is $W^{1,2}(D)$ and the chain rule of weak derivatives implies that
    \begin{equation*}
        \nabla(e^{\rho(x)}) = e^{\rho(x)} \nabla\rho(x)
    \end{equation*}
    holds for almost every $x$.
This already implies that many parts of our proofs carry over to the general case. As an example, Lemma~\ref{lem::convolveestimateconverges} follows almost immediately with the condition of precompactness in $W^{1,2}(D)$, or equivalently the uniform equicontinuity in $W^{1,2}(D)$, in the sense of Remark~\ref{rmk-uniform-equicontinuity}. However, there are a few technical issues where we need to improve our estimates. For example, our proof of Lemma~\ref{lem::convolveestimate} is based on a pointwise estimate of $\nabla \rho$, which would have to be replaced by an $L^2$ estimate, and such an estimate does not immediately follow from our current arguments.

\begin{ques}\label{ques-general-ftn} Prove that Theorem~\ref{thm::loopconvthm} holds for a larger class of functions. For example, one might consider functions in $W^{1,2}(D)$ and take $\mathcal B$ to be any precompact subset of $W^{1,2}(D)$, possibly with some additional conditions.
\end{ques}

Next, we note that Lemma~\ref{lem::convolveestimateconverges} is proven for Schwartz functions $\theta$, which is not expected to be optimal in any sense other than making the proofs simple. In fact, we believe the result holds for a much more general class of functions. For instance, we apply the lemma to \emph{expected} occupation measure, but it should also be true for a single occupation measure in some sense.

\begin{ques}\label{ques-occ} Prove that Lemma~\ref{lem::convolveestimateconverges} holds for a more general class of functions. For example, prove a similar result for a random function describing the occupation measure of a Brownian loop.
\end{ques}

Third, we may consider the case when the manifold $M$ has a boundary, say when $M=\BB H$.
In principle, one could formulate the boundary problem in a style similar to that presented above, and try to weaken the required regularity along the boundary as well.  To start, imagine the boundary line is the horizontal real axis.  Let $d \ell_b$ be the measure on unit length loops that hit the real axis and are centered at a point on the positive imaginary axis.

This measure can be obtained from by starting with $d\ell$, then {\em weighting} by the gap between the minimum and maximum imaginary values obtained by the loop, then shifting the vertical height loop to a uniformly random height (within this range).  The expected maximal height of a Brownian bridge is $\sqrt{\pi/8}$ (as can be proved using the reflection principle). The expected minimal height is thus minus that, and the expected difference $2\sqrt{\pi/8}$, which means that the expected vertical gap between the maximum height and the central height is again $\sqrt{\pi/8}$.

So formally, instead of being a probability measure, $d \ell_b$ is a measure with weight $\sqrt{\pi/8}$.  For real values $x$, write $\bclen(x, t, \ell_b) = e^{\rho(x)} \delta$. Note that the set of loops of length greater than $\delta$ is given by \begin{equation}\int_\delta^\infty \sqrt{\pi/8} \sqrt{t} \frac{1}{2\pi t^2}dt = \frac{2\sqrt{\pi/8} }{ 2\pi \sqrt{\delta} } = \frac{1}{2\sqrt{2\pi\delta}}\nonumber\end{equation}

We can imagine $\rho$ is defined in a neighborhood of a real line segment, and then try to measure the mass of loops (of size greater than $\delta$) that hit the boundary line itself.  The relevant quantities are the gradient in the parallel direction and the gradient in the normal direction (with the latter affecting whether the real axis is hit by more loops centered above or below).

\begin{ques}\label{ques-boundary} Prove that Theorem~\ref{thm-loop-general-surface} also holds for manifolds with boundaries. In other words, generalize the boundary case of~\cite[Theorem~1.3]{ang2020brownian}:\footnote{Similar to the case without boundary, we plug $\delta/2$ instead of $\delta$ in the statement of~\cite[Theorem~1.3]{ang2020brownian}, where $\delta$ there represents the quadratic variation of Brownian loops, which is two times the time interval length we use in this paper.}
\begin{conj}
    Let $(M,g)$ be a fixed compact smooth two-dimensional Riemannian manifold with smooth boundary, and we let $\mu^{\text{loop}}$ denote the Brownian loop measure on $(M,g)$. Let $K$ be the Gaussian curvature on $M$, let $\Delta$ be the Laplacian associated to $(M,g)$, and let $\det_\zeta' \Delta$ denote its zeta-regularized determinant. Let $\mcl B$ be a family of Lipschitz functions that (1) has uniformly bounded Lipshitz constants, and (2) is precompact in $W^{1,1}(M)$.
    
The $\mu^{\text{loop}}$-mass of loops with $\rho$-length greater than $\delta$ is given by
\begin{align}
\nonumber
&\frac{\Vol_{\rho}(M)}{2\pi\delta} - \frac{\Len_{\rho}(\partial M)}{2\sqrt{2\pi \delta}}
- \aconst (\log \frac{\delta}{2} + \upgamma) +  \frac{1}{48\pi} \int_M (\| \nabla \rho \|^2 + 2 K \rho)\,\Vol(dz)   
   \\ 
\nonumber
& \qquad\qquad \qquad
+ \log \Vol(M) - \log \Vol_\rho(M) - \log \det\nolimits_{\zeta}' \Delta
+ O(\delta^{1/2}), \label{eqn-p-a-boundary}
\end{align}
with the convergence as $\delta\to 0$ uniform over $\rho\in \mathcal B$, where $\upgamma \approx 0.5772$ is the Euler-Mascheroni constant.
\end{conj}
\end{ques}

Finally, one naive approach to generalizing~\cite[Proposition~6.9]{ang2020brownian} for lower-regularity settings is to follow the zeta-regularization procedure verbatim, hoping each step works for the generalized settings in a similar way.
The first obstacle in this direction is the lack of short-time expansions for the trace of heat kernels. For our application, what we need is if $(M,g)$ is a two-dimensional \emph{non-smooth} manifold without boundary, and $\Delta$ is the associated Laplacian defined in terms of Brownian loop mass, then
\[
\tr(e^{-\delta\Delta/2}) = \Vol_g(M)/\delta + \chi(M)/6 + o(1).
\]
\begin{ques}\label{ques-trace} Prove the short time expansion for the (trace) heat kernel holds for lower regularity metrics.
\end{ques}

\bibliographystyle{hmralphaabbrv}
\bibliography{cibib,cibib-alt,bibmore}
\end{document}